\documentclass{amsart}

\usepackage{amsmath,amsthm, amssymb,mathtools,mdwlist, MnSymbol}
\usepackage{mathrsfs}

\newtheorem{thm}{Theorem}[section]

\newtheorem{cor}{Corollary}[section]
   \makeatletter\let\c@cor\c@thm\makeatother

\newtheorem{prop}{Proposition}[section]
   \makeatletter\let\c@prop\c@thm\makeatother
   
\newtheorem*{prop*}{Proposition}

\newtheorem{lem}{Lemma}[section]
   \makeatletter\let\c@lem\c@thm\makeatother

\theoremstyle{definition}

\newtheorem{defn}{Definition}[section]
   \makeatletter\let\c@defn\c@thm\makeatother

\newtheorem{eg}{Example}[section]
   \makeatletter\let\c@eg\c@thm\makeatother

   \makeatletter\let\c@egs\c@thm\makeatother

   \makeatletter\let\c@ques\c@thm\makeatother

\theoremstyle{remark}
\newtheorem{rmk}{Remark}[section]
   \makeatletter\let\c@rmk\c@thm\makeatother

\makeatletter
\let\c@equation\c@thm
\makeatother
\numberwithin{equation}{section}

\usepackage[all]{xy}
\usepackage{hyperref}
\usepackage{tikz}
\usepackage{subcaption}

\usepackage{mathrsfs}
\newcommand{\op}{\mathscr{O}}
\newcommand{\mop}{\mathbb{O}}
\newcommand{\opt}{\mathscr{P}}
\newcommand{\mopt}{\mathbb{P}}
\newcommand{\amop}{\mathbb{T}}
\newcommand{\amopt}{\mathbb{T}'}

\newcommand{\cA}{\mathcal{A}}

\newcommand{\cG}{\mathcal{G}}
\newcommand{\cM}{\mathcal{M}}
\newcommand{\cN}{\mathcal{N}}

\newcommand{\cS}{\mathcal{S}}
\newcommand{\cT}{\mathcal{T}}
\newcommand{\cC}{\mathcal{C}}
\newcommand{\cW}{\mathcal{W}}

\newcommand{\cD}{\mathcal{D}}

\newcommand{\mA}{\mathbb{A}}
\newcommand{\mC}{\mathbb{C}}

\newcommand{\bvM}{\mathbf{M}}

\newcommand{\fa}[2]{{#1}({#2})}
\newcommand{\malg}{algebra}
\newcommand{\malgs}{algebras}
\newcommand{\oalg}{algebra}
\newcommand{\oalgs}{algebras}

\newcommand{\cube}{\mathscr{C}}

\let\catsymbfont\mathfrak 
\let\xto\xrightarrow
\newcommand{\aT}{{\catsymbfont{T}}}
\newcommand{\aU}{{\catsymbfont{U}}}
\newcommand{\aC}{{\catsymbfont{C}}}
\newcommand{\aD}{{\catsymbfont{D}}}

\DeclareMathOperator{\As}{Ass}
\DeclareMathOperator{\Com}{Com}
\DeclareMathOperator{\Diff}{Diff}

\DeclareMathOperator{\Id}{\mathrm{Id}}
\DeclareMathOperator{\id}{\mathrm{Id}}
\DeclareMathOperator{\eval}{eval}

\DeclareMathOperator{\sh}{sh}
\DeclareMathOperator{\inc}{incl}

\DeclareMathOperator*{\hocolim}{hocolim}
\DeclareMathOperator{\mAs}{\mathbb{A}ss}
\DeclareMathOperator{\proj}{proj}
\newcommand{\g}{\mathcal{G}}
\DeclareMathOperator{\hAut}{hAut}

\everymath{\displaystyle}

\begin{document}
\title[Homological stability and infinite loop spaces]{Infinite loop spaces from operads with homological stability}
\author[M. Basterra]{Maria Basterra}
\address{Department of Mathematics, University of New Hampshire,  Durham, NH 03824, USA}
\email{basterra@unh.edu}
\author[I. Bobkova]{Irina Bobkova}
\address{Department of Mathematics, University of Rochester, Rochester, NY 14620, USA}
\email{ibobkova@ur.rochester.edu}
\author[K. Ponto]{Kate Ponto}
\address{Department of Mathematics, University of Kentucky, Lexington, KY 40506, USA}
\email{kate.ponto@uky.edu}
\author[U. Tillmann]{Ulrike Tillmann}
\address{Department of Mathematics, University of Oxford, Oxford OX2 6GG, UK}
\email{tillmann@maths.ox.ac.uk}
\author[S. Yeakel]{Sarah Yeakel}
\address{Department of Mathematics, University of Maryland, College Park, MD 20742, USA}
\email{syeakel@math.umd.edu}

\begin{abstract}
  Motivated by the operad built from moduli spaces of Riemann surfaces,
  we consider a general class of operads in the category of spaces that satisfy certain homological stability conditions.  
  We prove that such operads are infinite loop space operads in the sense that the group completions of their algebras are infinite loop spaces. 
  
  The recent, strong homological stability results of Galatius and Randal-Williams for moduli spaces of  even  dimensional manifolds can be used to construct examples of operads with homological stability.  As a consequence diffeomorphism groups and  mapping class groups are shown to give rise to infinite loop spaces. Furthermore,
  the map to $K$-theory defined by the action of the diffeomorphisms on the middle dimensional homology is shown to be a map of infinite loop spaces.  
\end{abstract}

\date{\today}

\maketitle
\section{Introduction}

The theory of operads has its origin in the systematic study of loop spaces  \cite{BarrattEccles, boardmanvogt,May}.
Operads give a tool for the construction  and detection of loop spaces and, crucially,  facilitate  the study of maps between them. 
An important example is the little $n$-cube operad $\cube_n$ which detects 
$n$-fold loop spaces in the sense that the group completion of a $\cube_n$-algebra is an $n$-fold loop space, and in particular $\cube_\infty$-algebras group complete to infinite loop spaces. 
These are examples of $E_n$- and $E_\infty$-operads.

Let $\op = \{\op (j)\} _{j\geq 0}$ be an operad where each $\op(j)$ is a topological space equipped with a compatible action of the symmetric group $\Sigma_j$. In our setting, the space $\op(0)$ has a base point $*$ but may be much larger.
Because the study of infinite loop spaces is equivalent to the study of connective spectra (in a sense that can be made precise) $E_\infty$-operads are of particular interest. 
An operad $\op$ is an $E_\infty$-operad if each $\op (j)$ is contractible and has a free $\Sigma_j$-action.
In contrast, an operad $\op$ is an $E_1$-operad (also called $A_\infty$-operad) if each $\op(j)$ is homotopy equivalent to a free orbit of $\Sigma _j$. 

In general there are obstructions to extending a given $n$-fold loop space structure to an $n+1$-fold structure. Motivated by examples of operads that arise in the study of conformal and topological field theories, we prove here a very general theorem showing that the obstructions vanish in the presence of homological stability.

\begin{defn}
An operad $\op$ is an {\bf operad with homological stability (OHS)} if the following conditions are 
satisfied:
\begin{enumerate}
\item\label{ohs:condition1} the operad is graded, in the sense that each $\op(j)$ is the disjoint union of connected spaces 
$\op_g(j)$ for $g\in \mathbb N$ and the grading is compatible with the operad structure;
\item\label{ohs:condition2} the maps $D\colon \op_g(j) \to \op_g(0)$ induced by the operation of  $*\in \op(0)$ are homology isomorphisms in degrees $q \leq \phi(g)$ where $\phi (g) $ tends to infinity with $g$;
\item \label{ohs:condition3} there is an $A_\infty$-operad $\cA$ concentrated in degree zero, 
and a map of graded operads $\mu\colon \cA \to \op$ such that the image of $\cA(2)$ in $\op(2)$ is path-connected. 
\end{enumerate}
\end{defn}

We refer to  \S\ref{OHS} for a more general definition. 
The $\mathbb N$-grading of the operad is for convenience and is evident in our main examples, but we can also consider operads graded by other abelian monoids.  
The third condition ensures that any $\op$-algebra is a monoid up to homotopy with a group completion 
and the condition on $\mu(\cA(2))$ is a very mild commutativity condition that allows an application of the group completion theorem; see \S\ref{group_complete}. The second condition is the crucial one  which enables us to prove the following very general theorem.
Importantly, this condition can be weakened, see \S \ref{OHS}. We need only ask that the map $D$ induces a homology isomorphism on the homotopy colimits as $g$ goes to infinity.

\begin{thm}
Every operad $\op$ with homological stability  (OHS) is an infinite loop space operad: 
Group completion defines a functor  $\cG$ from the category of $\op$-algebras to the category of infinite loop spaces.
\end{thm}

A special case of this theorem was proved in \cite {T} for a discrete model of the surface operad. 
The above theorem allows us to simplify the model   
and extend the theorem to operads associated to non-orientable surfaces and surfaces with other tangential structures. We discuss this in the Appendix.
However, our main new examples are operads of higher dimensional manifolds.
We now give a brief description; the details are given in \S\ref{highdim}.
Let $W_{g,j+1}$ be the  connected sum of $g$ copies of $S^k \times S^k$  
with $j+1$ open disks removed. We will think of it as a cobordism from $j$ spheres of dimension $2k-1$ to one such sphere. As each $W_{g,j+1}$ is $k-1$-connected, the map to $BO(2k)$ that classifies the tangent bundle factors through the $k$-connected cover $\theta\colon B \to BO(2k)$.
We construct a graded operad $\cW^{2k}$ with 
\[
\cW^{2k}_g (j) \simeq  \mathcal M^\theta_k (W_{g, j+1}, \ell_{W_{g, j+1}})
\]
the connected component of the moduli space of $W_{g,j+1}$ with $\theta$-structure containing a given $\theta$-structure $\ell_{W_{g, j+1}}$, as defined and studied by Galatius and Randal-Williams. 
Analogously to the surface operad, the structure maps of $\cW^{2k}$ are defined by gluing the outgoing boundary sphere of one cobordism to an incoming sphere of another.  
Using the recent, strong homological stability results of \cite {GRWII} we 
prove the following.

\begin{thm}
For $2k \geq 2$, the operad $\cW^{2k}$ is an operad with homological stability (OHS).
\end{thm}

As an immediate consequence we 
construct infinite loop spaces from the classifying spaces of diffeomorphism groups of manifolds
and their mapping class groups. Furthermore, we prove that the map induced by the action of 
$\Diff (W_{g, 1}; \partial)$ on the middle dimensional homology $H_k (W_{g,1}) \simeq \mathbb Z^{2g}$ induces 
a  map of infinite loop spaces (or more precisely a zig-zag of such) from $\Omega ^\infty {\bf MT} ( \theta ) $
to $K$-theory.

\vskip .1in
The operad $\cW^{2k}$ can be considered as a subcategory of the $2k$-dimensional cobordism category as studied in \cite {GMTW} (after a slight modification). 
Thus the above theorem can be interpreted to say that for $2k \geq 2$, topological quantum field theories (TQFT) in topological spaces, i.e., symmetric monoidal functors from the $2k$-cobordism category to the category of spaces, take the $(2k-1)$-sphere to an $A_\infty$-monoid whose group completion is an infinite loop space. One may thus reasonably restrict attention to studying TQFTs with values in (connected) spectra.

\vskip .1in
Finally, we note some restrictions imposed by our definition of an OHS.
Above we have only mentioned new examples of operads for even dimensional manifolds. Even though homological stability has been extended to classes of odd dimensional manifolds by Perlmutter \cite{Perlmutter},  at the moment there is no homological stability result available that allows the number of  boundary components to be reduced (attaching a disk). However this is the homological stability that is needed for condition
\ref{ohs:condition2} of our definition.
The same remark holds for diffeomorphism groups for handlebodies which also have been shown to have homological stability by Perlmutter \cite{Perlmutter2}, and the stable homology been identified by  Botvinnik and Permutter \cite{BotvinnikPerlmutter}.
Similarly, Nariman \cite {Nariman1} proves homological stability results for diffeomorphism groups of manifolds made discrete. Interestingly, for surfaces he proves the weaker stable homology stability needed \cite {Nariman2}. However, for the discrete diffeomorphisms groups, condition \ref{ohs:condition3} also poses a challenge and  needs additional thought.

\subsection*{Outline}In \S\ref{sec:operads} we recall familiar constructions from the theory of operads and monads. We define 
a group completion functor for homotopy associative monoids and establish several important basic properties in \S\ref{group_complete}. 
In \S\ref{OHS} 
we define operads with homological stability  and in \S\ref{sec:freealgohs} we show that the 
group completions of free algebras over such operads are infinite loop spaces. 
We prove our main result in \S\ref{sec:infinite_loop} and show that
 \oalgs{} over operads with homological stability group complete to infinite loop spaces.
Our most important new example, the higher dimensional analog of the surface operad \cite {T}, is constructed and studied in \S\ref{highdim}
We describe various surface examples in the appendix. 

\subsection*{Acknowledgments}
The present paper represents part of the authors' Women in Topology project.  The other part \cite{BBPTY2} will be published elsewhere. 
We would like to thank the organizers of the Women in Topology Workshop for providing a wonderful opportunity for collaborative research and the Banff International Research Station for supporting the workshop. We also thank the Association for Women in Mathematics for partial travel support.


\section{Operads, Monads, and the Bar Construction}\label{sec:operads}

We recall some basic definitions and constructions from \cite {May}.

\vskip .2in

Let $\aU$ denote the category of compactly generated weakly Hausdorff spaces and continuous maps, 
and let $\aT$ denote the category of well pointed compactly generated Hausdorff spaces and based continuous maps.  

  An {\bf operad} is a collection of spaces \[\op=  \coprod_{n\geq 0} \op(n)\] 
  in $\aU$ with a nondegenerate base point $\ast\in \op(0)$, a distinguished element $1\in \op(1)$,
  a right action of the symmetric group $\Sigma_n$ on $\op(n)$  for each $n\geq 0$, and structure maps
  \[ 
   \gamma \colon  \op(k) \times  \op({j_1}) \times \ldots\times \op({j_k}) \to  \op(j_1+\ldots +j_k)
  \]
  for $k\geq 1$ and  $j_s \geq 0$.  
  The structure maps are required to be  associative, unital, and equivariant in the following sense:
  \begin{enumerate}
   \item For all $c\in \op(k)$, $d_s \in \op({j_s})$, and $e_t \in \op({i_t})$
    \[
     \gamma (\gamma(c;d_1,\ldots ,d_k);e_1,\ldots,e_j) = \gamma(c;f_1,\ldots,f_k)
    \] 
    where
    $f_s =\gamma(d_s;e_{j_1+\ldots +j_{s-1}+1},\ldots,e_{j_1+\ldots +j_s})$ 
    and $f_s =d_s$  if $j_s =0$.
   \item For all $d\in \op(j)$ and $c \in \op(k)$
    \[
     \gamma (1;d)=d \text{ and } \gamma(c;1^k)=c.
    \]
   \item  For all $c\in \op(k)$, $d_s \in \op({j_s})$, 
    $\sigma \in \Sigma_k$ and $\tau_s \in \Sigma_{j_s}$
    \[
     \gamma (c\sigma;d_1,\ldots ,d_k) = \gamma (c;d_{\sigma^{-1}(1)},
     \ldots ,d_{\sigma^{-1}(k)})\sigma (j_1,...,j_k)
    \] 
    and
    \[
     \gamma (c;d_1\tau_1,\ldots,d_k\tau_k)=
     \gamma(c;d_1,\ldots, d_k)(\tau_1 \oplus \ldots \oplus \tau_k)
    \]
    where $\sigma(j_1,\ldots,j_k)\in \Sigma_j$ permutes blocks of 
    size $j_s$ according to 
    $\sigma$, and $\tau_1 \oplus \ldots \oplus \tau_k$ denotes the image of
    $(\tau_1,\ldots ,\tau_k)$ under the natural inclusion of 
    $\Sigma_{j_1}\times \ldots \times \Sigma_{j_k}$ into  $\Sigma_j$.
  \end{enumerate}

While \cite {May} and most of the literature assumes the space $\op(0)$ of $0$-ary operations is a  single point, 
it is essential for us to allow more general spaces of $0$-ary operations.

\begin{eg}
  There are two discrete operads that play particularly important roles.  
  \begin{itemize}
   \item The $n$th space of the operad $\Com$ is a single point and all 
    structure maps are the unique map to a point.  This has trivial actions of the symmetric groups.
   \item The $n$th space of the operad $\As$ is $\Sigma_n$ and 
    the structure maps are given by block sum of permutations.
  \end{itemize}
\end{eg}

\begin{eg}[Little $n$-cubes operad, {\cite[\S 4]{May}}]\label{little_cubes}
  Let $I^n$ denote the unit $n$-cube and let $J^n$ denote its interior. 
  An {\bf (open) little $n$-cube}  is a linear embedding 
  $\alpha_1\times \ldots \times \alpha_n \colon J^n\to J^n$  
  where $\alpha_i\colon J^1\to J^1$ is defined   by $\alpha_i(t)=a_it+b_i$ for $0<a_i\leq 1$. 

  The $j$th space of the little $n$-cubes operad, $\cube_n(j)$, is the set of 
  $j$-tuples $\langle c_1,\ldots,c_j\rangle $  of little $n$-cubes where the images of the $c_r$
  are pairwise disjoint.  The topology  on $\cube_n(j)$ is induced from  
  the topology on the  space of continuous maps
  \[
   \coprod_jJ^n\to J^n.
  \]  
  The space $\cube_n(0)$ is a single point regarded as 
  the unique ``embedding'' of the empty set in $J^n$. 
  The structure maps are defined by composition of maps
  \[
   \gamma(c;d_1,\ldots,d_k) \coloneqq c\circ (d_1\amalg\ldots\amalg d_k) \colon 
   \left(\coprod_{j_1} J^n\right) \coprod \ldots \coprod\left(\coprod_{j_k}J^n\right)\to J^n
  \] 
  for $c\in \cube_n(k)$ and $ d_s\in \cube_n(j_s)$.  The identity object
  is $1\in \cube_n(1)$ and the action of $\sigma\in \Sigma_j$ is given by 
  $\langle c_1,\ldots,c_j\rangle \sigma=\langle c_{\sigma(1)}, \ldots,c_{\sigma(j)}\rangle$.
\end{eg}

\begin{eg}
  Let $\cM_{g,n}$ denote the moduli spaces of Riemann surfaces of genus $g$ with $n$ parametrized and ordered boundary components. 
  Segal \cite{Segal} constructed a symmetric monoidal category where the objects are finite unions of circles and 
  morphism spaces are  disjoint unions of the spaces $\cM_{g,n}$ with boundary circles 
  divided into incoming and outgoing.   Composition of morphisms is defined by gluing  outgoing 
  circles of one Riemann surface to incoming circles of another.  A conformal field theory in 
  the sense of \cite{Segal} is then a symmetric monoidal functor from this surface category 
  to an appropriate linear category.
 
  By restricting the category  we can define an operad  $\cM$ where 
  \[
   \cM(n)\coloneqq \coprod_{g\geq 0} \mathcal M_{g, n+1}.
  \]
  Note that $\cM (0)$ is non-trivial in this example. 
  The operad $\mathcal M$ contains a natural suboperad $\mathcal P$ 
  of Riemann surfaces of genus zero.  The operad $\mathcal{P}$ is 
  levelwise homotopic to the framed little two disk operad \cite[p. 282]{getzler}.
  
  We describe closely related operads in \S\ref{sec:examples}.
\end{eg}

A {\bf map  of operads} $\op \to  \opt$ is a collection of  $\Sigma_n$-equivariant maps 
$\op(n) \to  \opt(n)$ which commute with the structure maps and preserve $1$ and $*$.

\begin{eg}
The map $\phi_n\colon \cube_n\to \cube_{n+1}$ defined by
\[
 \phi_{n}(j)\left(\langle c_1,\ldots,c_j\rangle\right)\coloneqq 
 \langle c_1\times 1,\ldots,c_j\times 1\rangle \colon J^{n+1}\to J^{n+1}
\] 
is a map of operads.
Each $\phi_{n}(j)$ is an inclusion, and we let $\cube_\infty(j)$ denote the colimit of $\phi_{n}(j)$.  
These spaces assemble to an operad we denote by $\cube_\infty$.
\end{eg}

There are maps of operads $\cube_1\to \As$ and $\cube_\infty\to \Com$.  The 
first is a levelwise $\Sigma_n$-equivariant homotopy equivalence and 
the second is a levelwise homotopy equivalence \cite[3.5]{May}.  
Motivated by these examples, an operad $\op$ is an 
{\bf $A_\infty$-operad} if there is  a levelwise $\Sigma_n$-equivariant homotopy equivalence 
$\op\to \As$ and $\op$ is an {\bf $E_\infty$-operad} 
if the symmetric groups act freely levelwise and there is a levelwise homotopy equivalence $\op\to \Com$.  
This implies that an $E_\infty$-operad is levelwise contractible.

There are many examples where a direct map of operads is difficult to construct and we can only relate operads through a zig-zag of maps.  In these cases we use the product
of operads.

\begin{defn}The {\bf product} of operads $\op$ and $\opt$,   denoted $\op\times \opt$, is defined by 
  $(\op\times \opt)(j)\coloneqq \op(j)\times \opt(j)$  with the following structure:
  \begin{enumerate}
   \item for $c\times c'\in \op(k)\times \opt(k)$ and
    $d_s\times d_s'\in \op(j_s)\times \opt(j_s)$  
    \[
     (\gamma\times \gamma')(c\times c';d_1\times d_1',\ldots,d_k\times d_k')
     \coloneqq \gamma(c;d_1,\ldots,d_k)\times \gamma'(c';d_1',\ldots,d_k');
    \]
   \item $1 \coloneqq  1\times 1\in \op(1)\times \opt(1)$; 
   \item $* \coloneqq * \times * \in \op (0) \times \opt(0)$; and  
   \item  $(c\times c')\sigma \coloneqq c\sigma\times c'\sigma$ 
    for $c\times c'\in \op(j)\times \opt(j)$ and $\sigma\in\Sigma_j$.
  \end{enumerate}
\end{defn}

   For $A_\infty$-operads we refine this definition to operads over $\As$. 
   If $\cA$ and $\cA'$ are  $A_\infty$-operads, 
   $(\cA \times _{\As} \cA ') (n) $ is the pullback of the maps $\cA(n) \to \As (n) $ and 
   $\cA ' (n) \to \As (n)$. The structure maps are defined so that $\cA \times _{\As} \cA '$ 
   is a sub-operad of $\cA \times \cA '$.
   Then $\cA \times _{\As} \cA '$ is  an $A_\infty$-operad and there is a canonical map to $\cA \times \cA '$. 

  An {\bf $\op$-\oalg{}} is a  based space $(X, \ast)$ with  structure maps
  \[
   \theta\colon  \op(j) \times  X^j \to  X
  \]
  for all $j \geq  0$ such that
  \begin{enumerate}
   \item For all $c\in \op(k)$,  $d_s\in \op({j_s})$, and $x_t \in X$
    \[
     \theta(\gamma(c;d_1,\ldots,d_k);x_1,\ldots,x_j) = 
     \theta(c;y_1,\ldots,y_k)
    \]
    where $y_s = \theta(d_s;x_{j_1+\ldots+j_{s-1}+1},
    \ldots,x_{j_1+\ldots+j_s} )$.
   \item  For all $x \in X$
    \[
      \theta(1;x)=x\text{ and }\theta(\ast)=\ast.
    \]
   \item  For all $c\in \op(k)$, $x_s \in X$, $\sigma\in \Sigma_k$
    \[
      \theta (c\sigma;x_1,\ldots,x_k) = 
      \theta(c;x_{\sigma^{-1}(1)},\ldots ,x_{\sigma^{-1}(k)}).
    \]
  \end{enumerate}

A {\bf map of $\op$-\oalgs{}} is a map  $f \colon X \to  Y$ in $\aT$ that commutes with  the structure maps.
We will denote the category of $\op$-\oalgs{} by $\op[\aT].$

An operad map $\op\to \opt$ defines a $\op$-\oalg{} structure on any $\opt$-\oalg{}. Indeed, it defines a functor from
$\opt[\aT]$ to $\op[\aT]$.

\begin{eg}
The point $\ast$ is trivially an $\op$-\oalg{}, the  space $\op(0)$ 
is an $\op$-\oalg{} with structure maps
given by the operad structure maps, and   
every $\op$-\oalg{} $X$ receives a unique $\op$-\oalg{} map $\op(0) \to  X.$  
Additionally,  the map $X \to *$ is an $\op$-\oalg{} map for all $\op$-\oalgs{} $X,$ but   
the map $* \to X$ is an $\op$-\oalg{} map if and only if the map $\op(0) \to  X$ factors through $*.$ 
\end{eg}

For each space $X$ in $\aT$ we can define a free $\op$-\oalg{} on $X$. This
associates a monad to each operad allowing us to use familiar constructions from the theory of monads
in the context of operads \cite[\S 2]{May}.

  If $\op$ is an operad and $(X, \ast)$ is a based space,  the 
  {\bf free $\op$-\oalg{} on $X$} is 
  \[
    \fa{\mop}{X} \coloneqq  \coprod_{n\geq 0} \left(\op(n) 
    \times_{\Sigma_n} X^n\right)\big/ \sim 
  \]
  where $\sim$ is a base point relation generated by
  \[
    (\sigma_ic;x_1,\ldots ,x_{n-1}) \sim  (c;s_i(x_1,\ldots ,x_{n-1}))
  \]
  for all $c \in \op(n)$, $x_i \in X$, and $0 \leq i < n$ where 
  $\sigma_ic = \gamma(c,e_i)$ with 
  \[
    e_i =(1^i,\ast,1^{n-i-1})\in \op(1)^i \times \op(0) \times \op(1)^{n-i-1},
  \] 
  and $s_i(x_1,\ldots ,x_{n-1})=(x_1,\ldots ,x_i,\ast,x_{i+1},\ldots ,x_{n-1})$.
 
 The class of $(1,\ast) \in\op(1) \times X$ is the base point of $\fa{\mop}{X}$.
 Note that it coincides with the class of $\ast \in \op(0).$  

\begin{rmk}
We will occasionally abuse notation and identify points in $\mop(X)$ with their preimage in $\op(n)\times X^n$. 
In particular, we will identify elements of $\op(0)$ with their images in $\mop(X)$. This gives an identification of  $\op(0)$ and $ \mop(\ast) $.
\end{rmk}

The free algebra construction is functorial in $X$ and in $\op$. Any map of pointed spaces $ X\to Y$ induces 
a  map of $\op$-\oalgs{} $\fa{\mop}{X} \to \fa{\mop}{Y}$.  
Any map of operads $\op \to \opt$ gives rise to a map of $\op$-\oalgs{}  $\fa{\mop}{X} \to \fa{\mopt}{X}$.
The construction defines a monad in $\aT$ and 
associates a natural transformation of monads $\mop \to \mopt$ to a  map of operads  $\op \to  \opt.$ 
Hence, it defines a functor from operads in $\aU$ to monads in $\aT$. 

For a monad $\amop$ in a category 
$\aD$, a {\bf $\amop$-\malg{}}  is a pair $(X, \xi)$ 
consisting of an object $X$ in $\aD$ and a  map $\xi\colon \fa\amop X \to  X$ in $\aD$ that is unital and associative.  
A {\bf map of $\amop$-\malgs{}} is a map of spaces  $f\colon X \to Y$ where the following diagram commutes.
\[\xymatrix{
 \fa\amop X \ar[r]^{\fa\amop{f} }\ar[d]_{\xi_X}&\fa\amop Y\ar[d]^{\xi_Y}
  \\
 X\ar[r]^{f} & Y
}\]

The functor from operads to monads described above defines an isomorphism between the categories of $\op$-\oalgs{} 
and $\mop$-\malgs{}. The proofs of these facts in \cite[2.8]{May} carry over verbatim to 
the more general case  where $\op(0)$ is not necessarily a single point.

A {\bf $\amop$-functor} \cite[2.2, 9.4]{May} in a category $\aC$ is a functor $F \colon  \aD \to  \aC$ 
and a unital and associative natural transformation 
$\lambda \colon  F \amop \to F$.  
The {\bf bar construction} \cite[9.6]{May} for a monad $\amop$, $\amop$-\malg{} $X$ and $\amop$-functor $F$, denoted $B_\bullet(F, \amop, X)$, 
is the simplicial object in $\aC$ where $B_q(F, \amop, X)\coloneqq F\amop^q X$.
If $\aC$ is the category $\aT$ of based spaces, we define 
\[
  B(F, \amop, X)\coloneqq |B_\bullet(F, \amop, X)|.
\]

\begin{lem}\label{bar_construction_properties} \cite[9.7--9.10]{May}
 For $\amop$, $F$ and $X$ as above,
 the bar construction has the following properties:
 \begin{enumerate}
  \item For any functor $G \colon\aT \to \aT,$
  $GF$ is a $\amop$-functor in $\aT$ and there is a homeomorphism
   \[
     B (GF, \amop, X ) \cong |G B_{\bullet} (F, \amop, X)|.
   \]
  \item The structure map $\xi\colon \fa{\amop}{X} \to X $ induces a map of $\amop$-\malgs{} $B(\amop, \amop, X)
   \xto{\sim} X$ that is a strong deformation retract of spaces with inverse induced by the unit 
   $\eta_{\amop}\colon X \to \fa{\amop}{X}$.
  \item $F(\xi)$ induces a strong deformation retract of spaces
   $ B (F, \amop, \fa{\amop}{X})\xto{\sim} F(X)$. 
  \item If $\delta\colon \amop \to \amopt$ is a natural transformation of monads, then 
   $\amopt$ is a $\amop$-functor and  $B_{\bullet}(\amopt, \amop, X)$ is a simplicial $\amopt$-\malg{}.
   The composite 
   \[
     \tau\colon X \to B(\amop, \amop, X) \xrightarrow{B(\delta, 1,1)} B(\amopt, \amop, X)
   \] 
   where the first map is induced by the unit $\eta_{\amop}\colon X \to \fa{\amop}{X},$ 
   coincides with the map induced by the unit $\eta_{\amopt}\colon X \to \fa{\amopt}{X}$.
\end{enumerate}
\end{lem}

We finish this section with the observation that the bar construction for a monad is closely related to the classifying space construction.
Given a strict monoid $M$, we define an operad $\cM$ where 
\[
  \cM (0) =\ast,\cM(1) = M, \text { and } \cM (n) = \{ \} \text { for } n\geq 2.
\] 
The structure maps of the operad are given by the monoidal product and  $1$ is the monoidal unit.   
(Note that most of the structure maps for this operad are maps from the empty set to itself.)  $\cM$-algebras are spaces with an $M$-action.
The monadic bar construction for the trivial $\cM$-algebra $S^0$  coincides with the realization of the nerve of the category with a single object and morphisms given by  $M$ 
with an additional disjoint base point:
\[
  B(\Id, \mathbb{M}, S^0) = BM_+ \quad \text { where } \quad  BM = | N_\bullet M|.
\]


\section{Monoid Rectification and Group Completion}\label{group_complete}

For every $A_\infty$-operad $\cA$ we define a functor from $\cA$-algebras to groups up to homotopy 
that serves as the `group-completion' alluded to in our main theorem. Our first step is to functorially replace the monoid up to homotopy defined by the $A_\infty$-structure by a homotopy equivalent monoid
with a strictly associative multiplication. 
The second step is to
define a group completion for strictly associative monoids. 

\vskip .2in

\subsection{Monoid rectification}

Let $\cA$ be an $A_\infty$-operad and let  $\delta \colon \mA \to \mAs$ be the map of monads associated to the augmentation of operads $\cA \to \As$. 
For an $\cA$-\oalg{} $X$ define 
  \[ 
    \bvM_{\cA}(X) \coloneqq  B(\mAs,\mA,X).
  \]  
  
\begin{prop}
\label{assocthm}
  For any $A_\infty$-operad $\cA$, the construction  $ \bvM_{\cA}$ defines a functor from 
  the category of $\cA$-\oalgs{} to the category of strictly associative monoids. 
  For every $\cA$-\oalg{} $X$, there is  a strong deformation retract 
  \[
    \rho\colon  X\to \bvM_{\cA}(X)
  \]
  that is natural in $X$. Furthermore, $X$ and $\bvM_{\cA}(X)$  are related by a zig-zag of $\cA$-algebras. 
  Hence $\rho$ induces an isomorphism of monoids on  connected components and an isomorphism of homology Pontryagin rings.
\end{prop}

\begin{proof} 
  The simplicial space $B_{\bullet}(\mAs,\mA,X)$ is a simplicial $\mAs$-\malg{} 
  (i.e. a simplicial topological monoid),  and geometric realization preserves this structure \cite[9.10, 11.7]{May}.
  Thus $\bvM_{\cA}$ defines a functor from $\cA$-\oalgs{} to strict topological monoids. 
  
  \autoref{bar_construction_properties} and  properties of the bar construction give maps of $\cA$-\oalgs{} 
  \begin{equation}\label{rectification}
    X  \xleftarrow{\xi} B(\mA, \mA, X) \xrightarrow{B(\delta, 1,1)} B(\mAs, \mA, X)
  \end{equation}
  where the first map is a strong deformation retract of spaces with right inverse induced by the unit
  $\eta_{\mA}\colon X \to \fa{\mA}{X},$ 
  and the second is a homotopy equivalence  \cite[A.2, A.4]{May2}.  
  Let $\rho$ be the map $B(\delta,1,1)\circ\eta_{\mA}$.  
  This is the desired  homotopy equivalence of spaces which is natural in $X$.   
\end{proof}

\begin{lem}\label{comparemonoids} 
    Suppose $\cA$ and
     $\cA'$ are  $A_\infty$-operads and  $\epsilon\colon \cA \to \cA'$ is a map of operads so that the diagram 
     \[\xymatrix@C=8pt{
       \cA\ar[rr]^\epsilon\ar[dr]^\delta&&\cA'\ar[ld]_{\delta'}\\&\As
     }\] 
     commutes. 
     If $X$ is an $\cA'$-\oalg{} (regarded as an $\cA$-\oalg{} via $\epsilon$), there is a map of monoids 
     $\bvM_\cA(X)\to  \bvM_{\cA'}(X)$ that is a homotopy equivalence.
\end{lem}
In particular, rectification is independent of the choice of $A_\infty$-operad.  
When clear from context, we will drop the subscript from the notation for the rectification functor.

\begin{proof}
  There is a commutative diagram 
  \[\xymatrix@C=40pt@R=5pt{
     &B(\mA, \mA, X)\ar[r]^-{B(\delta, 1,1)}\ar[ld]_-{\xi} \ar[dd]^{B(\epsilon,\epsilon,1)}
    &B(\mAs, \mA, X)\ar[dd]^{B(1,\epsilon,1)}
    \\
    X
    \\
    &B(\mA', \mA', X) \ar[r]^-{B(\delta', 1,1)} \ar[ul]^-{\xi'}&B(\mAs, \mA', X)
  }\]
  where all arrows are homotopy  equivalences and $B(1, \epsilon,1)$ is a map of $\As$-algebras.
\end{proof}

The functoriality of the rectification functor is well-understood and we leave the proof of the following result as an exercise.

\begin{lem}\label{rectproperties} Let $\cA$ and $\cA'$ be $A_\infty$-operads.
 \begin{enumerate}
  \item\label{rectproduct} If $X$ and $X'$ are $\cA$- and $\cA'$-algebras, then $X\times X'$ is a $\cA \times_{\As} \cA'$-algebra and
    \[
    \bvM_{\cA \times_{\As} \cA'} (X\times X') \simeq \bvM_{\cA}( X) \times \bvM_{\cA'} (X').
    \]
  \item\label{rectrealization} If $X_\bullet$ is a simplicial $\cA$-algebra then its realization is an $\cA$-algebra and 
   \[
     \bvM_{\cA} |X_\bullet| \simeq |\bvM_{\cA} X_\bullet| . 
   \]
 \end{enumerate}
\end{lem}

\subsection{Constructing a group completion}
For an algebraic monoid, group completion is defined by formally inverting elements; for a topological monoid, a natural analog is to use $\Omega B$, the based loops of the bar construction.
Motivated by this observation, we define 
the {\bf group completion} of a strictly associative topological monoid $M$ to be 
\[
  \g M \coloneqq \Omega BM
\]
where $BM=|N_\bullet M|$, the geometric realization of the nerve of $M$ considered as a category with one object.
The identification of $M$ with the 1 simplices in the nerve $N_\bullet M$ 
extends to  a canonical map  $[0,1] \times  M \to BM$ which factors through $S^1 \wedge M$. Taking its adjoint defines a natural map 
\[g\colon M\longrightarrow \Omega BM = \g M.\]

\begin{thm}
\label{gcthrm}
  The assignment  $M \mapsto \g M$ defines a functor $\g$ from well-pointed strictly associative topological monoids 
  to $\Omega $-spaces and $g\colon \Id \to \g$ is a natural transformation. 
  
  Additionally, group completion has the following 
  properties:
  \begin{enumerate}
    \item \label{gc:product}
     For monoids $M$ and $M'$,
     $\g(M \times M') \simeq \g(M) \times \g(M').$ 
    \item\label{gc:grouplike} If $M$ is a  grouplike monoid, then
     $g\colon M \to \g(M)$ is a weak equivalence.
    \item \label{gc:equivalence}
     If a map of monoids $M \to M'$ is a homotopy equivalence, then $\g M \to \g M'$ is a homotopy equivalence.
    \item \label{gc:monoid} If $M_{\bullet}$ is a simplicial monoid, 
       then    $|\g M_{\bullet}|\simeq \g|M_{\bullet}|.$
 \end{enumerate}
\end{thm}

  We assume  that all our monoids are well-pointed in the sense of \cite [A.8] {May}.
  This will ensure that the bar construction is a `good' simplicial space in the sense of \cite [A.4]{Segal}.
  Thus the fat realization agrees with the usual realization and levelwise homotopy equivalences 
  induce homotopy equivalences on total spaces.

\begin{proof} [Proof of \autoref{gcthrm}]
  These are standard facts. Assertion (\ref{gc:product}) is an immediate consequence of the fact that 
  the  nerve  $N_\bullet (-)$, geometric realization $| - |$, and taking loops $\Omega (-)$ 
  are all functors that take products to products. 
  For (\ref{gc:grouplike}) see, for example, McDuff-Segal \cite[1] {gp}.
  For (\ref{gc:equivalence}) note that the induced  map of nerves   $N_\bullet M \to  N_\bullet M'$ 
  is a levelwise homotopy equivalence, and hence induces a homotopy equivalence on realizations.  
  Finally for (\ref{gc:monoid}), note that for any simplicial monoid $M_{\bullet}$,  
  \[
    |B M_{\bullet}| \cong B |M_{\bullet}|
  \] 
  as the order in which the realization of bisimplicial objects is taken can be interchanged. 
  Furthermore,  as each $B M_n$ is connected, taking based loops $\Omega $ 
  commutes with the realization functor, compare \cite[12.3]{May}, 
  and hence  $|\Omega B M_{\bullet}| \cong \Omega|B M_{\bullet}|$ and so
  \[
    |\Omega B M_{\bullet}| \cong \Omega B |M_{\bullet}|.\qedhere
  \] 
\end{proof}

We extend the above discussion to $A_\infty$-algebras by first applying the rectification functor. Thus, given 
an $A_\infty$-operad $\cA$ and an $\cA$-\oalg{} $X$  define 
\[
  \g (X) \coloneqq \Omega B (\bvM _{\cA}(X)) = \Omega B (B(\mAs, \mA, X)),
\]
and let $g\colon X \to \g(X)$ be the composite 
\[ 
  X \xto{\rho}  B(\mAs, \mA, X)=\bvM_{\cA}(X)\longrightarrow \Omega B \bvM_{\cA}(X)=\g (X).
\]  
With these definitions, group completion defines a functor from the category of $\cA$-algebras 
to the category of loop spaces, and \autoref{gcthrm}  
remains valid. Indeed, (\ref{gc:product}) follows from \autoref{rectproperties}(\ref{rectproduct}), (\ref{gc:grouplike}) and (\ref{gc:equivalence})
hold since weak equivalences are preserved by the monoid rectification functor $\bvM _{\cA}$, and (\ref{gc:monoid}) is a consequence of the observation that  $\bvM_{\cA} $ commutes up to homotopy with realization for simplicial $\cA$-algebras in \autoref{rectproperties}(\ref{rectrealization}).

\subsection{Group completion theorem}

 We recall the group completion theorem from \cite {Quillen}; also see  \cite {gp,RandalWilliams}.

\begin{thm}\label{thm:gpcompletion}
  Let $M$ be a well-pointed topological monoid.  If $\pi_0 (M)$ is in the center of $H_*(M; k)$ then 
  \[
    {g}_*\colon H_{*}(M;k)[(\pi_0 M)^{-1}]\to H_*(\g M;k)
  \]
  is an isomorphisms of $k$-algebras for all commutative coefficient rings $k.$
\end{thm}

\begin{rmk}
  This result also holds under the hypothesis that 
  $H_*(M;k) [\pi_0(M) ^{-1}]$ can be constructed by right fractions \cite {Quillen,RandalWilliams}.
\end{rmk}

If $\pi_0(M)$ is finitely generated by $s_1, \dots , s_n$, the localization 
can be achieved by inverting the product $s\coloneqq s_1\dots  s_n$:
\[
  H_{*}(M;k)[(\pi_0 M)^{-1}] =  H_{*}(M;k)[s^{-1}]. 
\]
The right term can be identified with the homology of  the telescope 
\[
  \hocolim_ {\tilde s }M = \text {Tel} (M \xto {\tilde s}  M \xto {\tilde s} M \xto {\tilde  s }  \dots)
\]
where $\tilde s\in M$ is in the connected component corresponding to $s$.
Thus, if we let 
\[
  M_\infty \coloneqq (\hocolim _{\tilde s} M)_{\id}
\]
denote the connected component of the identity of $ \hocolim_ {\tilde s }M$, the homology of a connected component of the group completion is  given by
\begin{equation}\label{gpcompletioniso}
  H_*((\g M)_{\id} ; k) = H_*(M_\infty; k).
\end{equation}

If the monoid is homotopy commutative, there is a refinement of the above theorem which allows for abelian coefficient systems.
We thus have the following corollary, compare  \cite[1.2]{RandalWilliams}.

\begin{cor}
  Let  $M$ be a homotopy commutative monoid with $\pi_0 (M)$ finitely generated  
  and $M_\infty$ defined as above. 
  Then the commutator subgroup of the fundamental group 
  $\pi_1 (M_\infty)$ is perfect and there is a weak homotopy equivalence
  \[
    (\g M)_{\id} \simeq M_\infty ^+,
  \]
  where $Z^+$ denotes Quillen's plus construction.
\end{cor}
 
Because the Pontryagin homology ring of an $A_\infty$-algebra is isomorphic to that of its rectification, the discussion above immediately also extends to any such  monoid up to homotopy. 

\begin{rmk}
 Any $\cube _2$-algebra (or more generally any $E_2$-algebra) is  
 homotopy commutative and hence the conditions of the group completion theorem, \autoref{thm:gpcompletion}, 
 and its corollary are immediately satisfied.
\end{rmk}


\section{Operads with homological stability (OHS)}\label{OHS}
 
In this section we define the operads of primary interest. In particular, they are graded
and receive a map from an $A_\infty$-operad.
 
\begin{defn}\label{gradedV}
  Let $I$ be a commutative, finitely generated monoid with identity element $0$.  
  An {$I$-\bf grading} on an operad $\op$ is a decomposition 
  \[
    \op(n)=\coprod_{g \in I}\op_{g}(n)
  \] 
  for each $n$ so  that:
  \begin{enumerate}
    \item the basepoint $*$ lies in $ \op_0(0)$,
    \item the $\Sigma_n$ action on $\op(n)$ restricts to an action on each $\op_{g}(n)$, and
    \item the structure maps restrict to maps 
      \[
        \gamma \colon  \op_g(k) \times  \op_{g_1}({j_1}) \times \ldots
        \times \op_{g_k}({j_k}) \to  \op_{g+g_1+\ldots +g_k}(j_1+\ldots +j_k).
      \]
 \end{enumerate}
\end{defn}

In particular, $1 \in \op_0 (1)$ and the maps 
\[
  \sigma_i\coloneqq \gamma(-; 1^i, \ast, 1^{n-i-1}) \colon \op(n) \to \op(n-1)
\] 
restrict to  maps $\sigma_i\colon \op_g(n) \to \op_g(n-1)$.

We will mainly be concerned with $\mathbb N$-graded operads. 
Note that any ordinary operad is an $I$-graded operad concentrated in  degree 0.  Furthermore, 
if $J$ is a submonoid of $I$ then the graded pieces of $\op$ corresponding to $J$  form a suboperad of $\op$.

Let $\op$ be an $I$-graded operad and assume there is   a map of graded operads 
$\mu \colon \cA\to \op$ for some $A_\infty$-operad $\cA$.
 We refer to $\mu$ as a {\bf multiplication map.}  
Then every $\op$-\oalg{}   is an $A_\infty$-algebra and
we can apply the group completion functor defined in the previous section. 
Note  that the image of $\cA$ is contained 
in the degree 0 part of $\op$, and for most of the remainder of this paper we will assume that $\op$ 
satisfies the following {\bf weak homotopy commutativity condition}: 
\begin{equation}\label{htpycom}
  \mu (\cA (2)) \subset \op_0(2) \text { is contained in a path component.}
\end{equation}

\begin{rmk}\label{htpycom_detail}
  Condition \ref{htpycom} implies  there is a path between any two points in the image of  $\mu$.  Thus 
  the induced multiplication  on any $\op$-algebra 
  is homotopy commutative and  the group completion theorem, \autoref{thm:gpcompletion}, and its
  corollary apply.
\end{rmk}

Let  $D\colon \op(n)\to \op(0)$ be defined by $D(-)=\gamma(-;\ast, \ldots,\ast)$. 
For each $n\in\mathbb{N}$, $g\in I$ and $\tilde s\in \op _s(1)$, the diagram 
\[\xymatrix{
 \op_{g}(n)\ar[r]^-{\tilde s}\ar[d]_{D}&\op_{g+s}(n)\ar[d]^{D}\\
 \op_{g}(0)\ar[r]^-{\tilde s} &\op_{g+s}(0)
}\]
commutes by the associativity of $\gamma$.
Here we abuse notation and use $\tilde s$ to denote the map $\gamma (\tilde  s; -)$. 

Now let $s$ be the product of a set of  generators for $I$ and choose $\tilde{s}\in \op_s(1)$.  
We will refer to both the element $\tilde{s}$ and the map $\gamma(\tilde{s},-)$ as 
the {\bf propagator}. Define 
\[
  \op_\infty(n) \coloneqq \hocolim_{\tilde s} \op_g(n).
\]
The commutative square above defines a map 
\[
  {D}_\infty \colon \op_\infty(n) \longrightarrow  \op_\infty (0).
\]
  Note that any other choice of $\tilde s$ in the same path component of $\op_s(1)$ will 
  produce a homotopic map, but a choice from a different component may not.

\begin{rmk}\label{rmk:propagator2}
  In later sections it will be convenient to view the map $\tilde{s}$ as 
the  multiplication with an element $\tilde{s}_0 \in\op_s(0)$.  For this purpose 
  we construct $\tilde{s}$ as follows:  
  let $\mu_0$ be an element in the path-connected image $\mu(\cA (2)) \subset \op_0(2)$ and  pick $\tilde s_0 
    \in \op_s(0)$.  Then define 
   \[ \tilde s \coloneqq  \gamma (\mu_0; \tilde s_0 , 1) \in \op_s(1). 
  \] 
  For any $\op$-algebra $X$,
  the map 
    $\theta(\tilde{s};-)\colon X\to X$ is homotopic to the multiplication 
  by the image of $\tilde{s}_0$ under the map $\theta\colon \op(0)\to X$. 
\end{rmk}

\begin{defn}\label{HS}
  Let $\op$ be an $I$-graded operad together with a map $\mu\colon \cA \to \op$ 
  that satisfies condition \ref{htpycom}.
  Then $\op$  is  an {\bf operad with homological stability (OHS)} if  the maps 
  \[
    D_\infty\colon \op_\infty (n) \longrightarrow \op_\infty (0)
  \]
  induce homology isomorphisms. 
  
For operads $\op$ and $\opt$ with homological stability,  
  a {\bf map of operads with homological stability} is a pair of graded operad maps $f\colon \op \to \opt$ 
  and $a\colon \cA \to \cA'$ such that the following diagram commutes:
  \[\xymatrix{
   \cA\ar[r]^-a\ar[d]_{\mu}&\cA'\ar[d]^{\mu '}\\
   \op\ar[r]_-f&\opt.
  }\]
\end{defn}

    For $I=\mathbb N$,  $D_\infty$  is a homology isomorphism if the maps \[D\colon \op_g (n) \to \op_g(0)\]
    are homology isomorphisms  in degrees $* < \phi (g)$  where $\phi$ goes to infinity as $g$ goes to infinity.
    For operads concentrated in degree 0,
    $\op_\infty (n) \simeq \op(n)$ and the condition is that $D\colon \op (n) \to \op (0)$ is a homology isomorphism. 

\begin{eg}
  $\cube _\infty$ is an OHS with 
  multiplication  $\mu\colon \cube_1 \to \cube_\infty$ defined by the inclusion.  
  Note that $\mu$ factors through $\cube _2$ and so satisfies condition \ref{htpycom}.   
  (In fact, it is enough to note that $\cube_{\infty}(2)$ is contractible.)
  As all levels $\cube _\infty (n)$ are contractible, the 
  homological stability condition is trivially satisfied.
\end{eg}

\begin{eg}
 The Riemann surface operad $\mathcal M$  receives a map from the genus zero operad $\mathcal P$ and 
  hence it satisfies condition \ref{htpycom}.
  
  By Harer's homological stability theorem, the homology of the mapping class groups 
  $H_*(\Gamma _{g, n+1})$ are independent of $g$ and $n$ as long as $*  $ is small relative to $g$ \cite {Harer}. Since $\cM _g (n) \simeq B\Gamma _{g, n+1}$, 
  the maps $D\colon \cM _g(n) \to \cM _g (0)$ are homology isomorphisms in degrees
  $* $ growing with $g$. 
  
  In particular, $\mathcal M$ is an operad with homological stability.
\end{eg}

Taking  the operadic product $\op \times \opt$ of operads $\op$ and $\opt$ is functorial 
in both factors and the projection maps allow us to consider $\op$-\oalgs{}  and $\opt$-\oalgs{} 
as $(\op \times \opt)$-\oalgs{}.  We can upgrade the product of operads to a product in the category of OHSs.

\vskip .2in
\begin{lem}\label{lem:prodohs}
  Let $\op$ and $\opt$  be two  OHSs. Then the product operad $ \op \times \opt$  
  is an OHS with compatible maps of OHSs
  \[
    \op \xleftarrow{\pi_1}  \op\times \opt \xto {\pi_2}   \opt
  \]
  where $\pi_1$ and $\pi_2$ are the projections to the first and  second factors, respectively. 
\end{lem}

\begin{proof}
  If $\op$ is $I$-graded and $\opt$ is $I'$-graded, the product $\op\times \opt$ is $I\times I'$-graded  
  with graded pieces
  \[
    (\op \times \opt)_{(g, g')}(n) = \op_g (n) \times \opt_{g'} (n).
  \]
  The symmetric group $\Sigma_n$ acts diagonally.

  Let $\mu \colon \cA \to \op$ and $\mu'\colon \cA' \to \opt$ be the multiplications for $\op$ and 
  $\opt$. 
  Then $\cA\times_{\As}\cA'$  is an $A_\infty$-operad and the composite
  \[
    \tilde \mu \colon  \cA\times_{\As}\cA'  \longrightarrow \cA \times \cA' \xto {\mu \times \mu'}
    \op \times \opt
  \]
  defines a multiplication map for $\op \times \opt$.
  Since $\op$ and $\opt$ satisfy condition \ref{htpycom}, $\op \times \opt$ also satisfies this condition.

  Suppose $I$ is generated by $s_1,\ldots, s_n$ and $I'$ is generated by $s_1',\ldots, s_m'$.  Let $\tilde s=s_1 \times \ldots \times s_n$ and $\tilde s'=s_1' \times \ldots \times s_m'$, then $(\tilde s,\tilde s')$ is a generator for $I\times I'.$
  Thus 
  $\hocolim_{(\tilde s,\tilde s')} (\op \times \opt)(n) \simeq \hocolim _{\tilde s} \op (n)
  \times \hocolim _{\tilde s'} \opt(n)$ 
  and  we have 
  \[
    D_\infty \times D'_\infty \colon 
    \op_\infty (n) \times \opt_\infty (n) \longrightarrow
    \op_\infty (0)\times \opt_\infty (0)
  \]
  is a homology isomorphism.  Then $\op \times \opt$ is an OHS.
\end{proof}

In the next sections we will focus on operads  $\op$ with homological stability  that admit a map of OHSs  
\[
  \pi \colon \op \to \cube_\infty.
  \footnote{ $\cube_\infty$ is a convenient  $E _\infty$-operad, but any other $E_\infty$-operad could be used.}
\]
Any element  $\tilde c \in \cube_\infty(1)$ is a propagator for $\cube_\infty$, and we have that
\begin{equation}\label{trivial} 
  \hocolim_{\tilde c} \cube_\infty (n) \simeq \cube_\infty (n). 
\end{equation} 
In particular if $\tilde s$ is the propagator for the OHS $\op,$ then $\pi(\tilde s)$ is a 
propagator for $\cube_\infty.$  To simplify notation and in light of \eqref{trivial} we write $\cube_\infty(n)$ 
for $\hocolim_{\pi(\tilde s)} \cube_\infty (n)$ and we have maps $\pi_\infty \colon  \op_\infty (n) \to  \cube_\infty (n)$.

The following result allows us to replace $\op$-algebras with $\widetilde{\op}$-algebras, where $\widetilde \op$ is an operad with an OHS map to $\cube_\infty$.

\begin{cor}\label{cor:productohs}
  Let $\op$  be an  OHS. Then the product operad $\widetilde \op \coloneqq \op \times \cube _\infty$  
  is an OHS with compatible maps of operads
  \[
    \op  \xleftarrow{\pi_1} \widetilde {\op} \xto \pi   \cube_\infty
  \]
  where $\pi_1$ and $\pi$ are the projections to the first and  second factors respectively. 

  Indeed, $\op \mapsto \widetilde \op$ defines a functor from the category of OHSs to the category of OHSs over  $\cube_\infty$.
\end{cor}

\begin{proof}
  This follows immediately  from \autoref{lem:prodohs} since $\cube _\infty$ is an OHS. 
\end{proof}

Thus any $\op$-\oalg{} is an $\widetilde \op$-\oalg{}, and the statements and properties we prove for $\widetilde \op$-\oalgs{} will equally apply to $\op$-\oalgs{}. Likewise, any $\cube_\infty$-\oalg{} is a $\widetilde \op$-\oalg{}, and for each $n \geq 0$, the symmetric group $\Sigma_n$ acts freely on $\widetilde \op (n)$.

\section{The group completion of free \oalgs{} over an OHS}\label{sec:freealgohs}

The goal of this section is to describe the group completion of a free
$\op$-\oalg{} for an OHS with a compatible map $\pi\colon \op \to \cube_\infty$. (Recall from the previous section that there is no loss of generality in assuming we have the map $\pi$.) 
We will show that, up to homotopy, the group completion of the free  $\op$-\oalg{} on a based space $X$ is a product of the group completion of $\op(0)$ and  the free infinite loop space on $X$.
This will be used in the following section where we replace $X$ by its free resolution as an $\op$-algebra.
First however, we will prove \autoref{productstability} which uses the homological stability condition in the definition of an OHS in an essential way. 

For each $n$ and $\Sigma_n$-space $Y$, we have $\Sigma_n$-equivariant maps
 \[ 
   {D}_\infty \times \ast\colon \op_\infty(n) \times Y \to \op_\infty (0) \times \ast
 \] 
inducing  maps $\bar {D}_\infty\colon \op_\infty (n)\times_{\Sigma_n} Y\to \op_\infty(0)$.
Taking the product with  the compatible map of operads $\pi \colon \op\to  \cube_\infty$  gives a  map
\begin{equation}\label{splitmaps}
  \bar {D}_\infty \times (\pi_\infty \times 1_Y) \colon \op_\infty (n)\times_{\Sigma_n} Y
  \to \op_\infty (0) \times (\cube_\infty(n) \times_{\Sigma_n}Y).
\end{equation}

\begin{lem}
  \label{productstability}
  Suppose $\op$ is an OHS and there is a compatible map  $\pi\colon \op \to \cube_{\infty}$. Then 
  the induced map 
  \[
   H_*(\op_{\infty}(n)\times_{\Sigma_n} Y)\to 
   H_*(\op_\infty(0) \times (\cube_\infty(n) \times _{\Sigma_n}Y))
  \] 
  is an isomorphism for all $\Sigma_n$-spaces $Y$.
\end{lem}

\begin{proof} 
  The space $\cube_\infty(n)$ is contractible,  and so the maps 
  \begin{equation}\label{product_cubes}
    \op_{\infty}(n)\xto{1\times \ast_Y} \op_{\infty}(n)\times Y
    \xto{\pi_\infty\times 1_Y} \cube_{\infty}(n)\times Y
  \end{equation} 
  induce a long exact sequence on homotopy groups.  
  Since $\pi_\infty$ is $\Sigma_n$-equivariant and the actions of $\Sigma_n$ on $\op_{\infty}(n)$ 
  and $\cube_{\infty}(n)$ are free, the long exact sequence in homotopy groups for  \eqref{product_cubes} 
  implies the maps 
  \[
   \op_{\infty}(n)\xto{1\times \ast_Y} \op_{\infty}(n)\times_{\Sigma_n}Y
   \xto{\pi_\infty\times 1_Y} \cube_{\infty}(n) \times_{\Sigma_n}Y
  \]
  induce a long exact sequence on homotopy groups.  

  Consider the following diagram where the second and third rows are 
  fibrations and the first row induces a long exact sequence on homotopy groups.  
  \[\xymatrix@C=35pt{
    \op_{\infty}(n)\ar[r]^-{1\times \ast_Y}
    &\op_{\infty}(n)\times_{\Sigma_n}Y\ar[r]^{\pi_\infty\times 1_Y}\ar[d]^{(1\times \pi)\times 1}
    &\cube_{\infty}(n)\times_{\Sigma_n}Y\ar@{=}[d]
    \\
    \op_{\infty}(n)\ar[r]^-{1\times \ast_Y}\ar[d]_{{D}_\infty}\ar@{=}[u]
    &(\op_{\infty}(n)\times \cube_\infty(n))\times_{\Sigma_n}Y
    \ar[r]^-{\proj_2\times 1_Y}\ar[d]^{\bar {D}_\infty\times   1\times 1_Y}
    &\cube_{\infty}(n)\times_{\Sigma_n}Y\ar@{=}[d]
    \\
    \op_\infty(0)\ar[r]^-{1\times \ast}
    &\op_\infty(0)\times(\cube_{\infty}(n)\times_{\Sigma_n}Y)\ar[r]^-{\proj_2}
    &\cube_{\infty}(n)\times_{\Sigma_n}Y
  }\]
  The top left square commutes up to homotopy and the remaining 
  small squares all commute strictly.  
  
  As the $\Sigma_n$ action on $\op_\infty(n)$ is free, the map 
  \[ 
   \op_\infty(n)\times _{\Sigma_n}Y\to (\op_\infty(n)
   \times \cube_\infty(n))\times _{\Sigma_n}Y
  \]
  is a weak homotopy equivalence.  The map between the bottom
  two rows defines a map of fibrations and hence  a map of Serre spectral sequences.  
  On the $E_2$-page, the map
  \[
   H_p(\cube_{\infty}(n)\times_{\Sigma_n}Y;H_q(\op_{\infty}( n)))
   \to H_p(\cube_{\infty}(n)\times_{\Sigma_n}Y;H_q(\op_{\infty}(0)))
  \]
  is an isomorphism since ${D}_\infty$ induces an isomorphism in homology by the assumption that $\op$ is an OHS.  
  This induces an isomorphism on the $E_\infty$-page and hence the desired isomorphism between the homology of the total spaces \cite[3.26]{weibel}.
\end{proof}

As we assume we have a compatible map of operads $\pi \colon \op \to \cube_\infty$, 
any $\cube_\infty$-algebra is an $\op$-algebra. For any space $X,$ there is a map  of $\op$-algebras
\[
 \tau \times \pi\colon  \fa{\mop}{X} \to \fa{\mop}\ast \times\fa{\mC_{\infty}}X ,
\]
where the map $\tau$ is induced by the unique map $X \to \ast$ and  the diagonal action of $\op$ provides 
the $\op$-algebra structure on the product $\fa\mop\ast\times \fa{\mC_\infty}X$.
Being maps of $\op$-algebras,  $\tau$ and $\pi$ are maps of $\cA$-\oalgs{}, 
so they induce a map on group completions
 \[ 
  \g(\tau)\times \g(\pi)\colon \g \fa\mop{X} \to  \g(\fa\mop\ast)\times \g (\fa{\mC_{\infty}}X). 
 \] 
We can now state the main result of this section.  

\begin{thm}
  \label{split}
  Suppose  $\op$ is an OHS with a compatible operad map  $\pi\colon \op \to \cube_{\infty}$.
  For  any based space $X$,  
  \[
   \g(\tau) \times \g( \pi)\colon \g( \fa\mop{X})\to  \g(\fa\mop\ast)\times  \g(\fa{\mC_{\infty}}X)
  \] 
  is a weak homotopy equivalence.
\end{thm}

\begin{proof}
  For any based space $X$, the grading on $\op$ 
  defines a decomposition of  $\fa\mop{X}$ as $\coprod_{g \in I} \fa{\mop_g}{X}$ where 
  \[ 
   \fa{\mop_g}X \coloneqq  \left(\coprod_n \op_{g}(n)\times_{\Sigma_n} X^n\right)/\sim  
  \]
  and the equivalence relation is as in the definition of the free $\op$-\oalg{}. 
  
  Similarly, we have a natural decomposition of $\fa{\mop}\ast \times \mC _\infty (X)$
  with graded pieces 
  \[
   \fa{\mop _g} \ast \times \mC_\infty (X).
  \]
  Left multiplication by the propagator defines maps 
  \[
   \tilde s  \colon \fa{\mop_g}X \to \fa{\mop_{g+s}}X \quad  \text { and } \quad 
   \tilde s \times 1\colon \fa{\mop_g} \ast \times \mC _\infty (X) \to \fa{\mop_{g+ s}} \ast \times \mC_\infty (X) 
  \] 
  Let 
  \[
   \fa{\mop_\infty}X \quad \text{ and } \quad \fa{\mop _\infty } \ast \times \mC _\infty (X)
  \] 
  denote the homotopy colimits over these maps. 
  Then, under the homeomorphism  $\fa{\mop_g}\ast \cong \op_g(0)$,
  \autoref{rmk:propagator2} shows $\fa{\mop_\infty}\ast \cong \op_\infty (0)$.
  
  It will be enough to show that the  map
  \begin{equation}\label{limitmap}
   \tau_\infty \times \pi_\infty \colon \fa{\mop_\infty}X \longrightarrow\fa{\mop _\infty } \ast \times \mC _\infty (X)
  \end{equation}
  is a homology isomorphism. Since, if we let $s$ denote the component of $\tilde s$, 
   it will follow that
  \[
   \tau \times \pi\colon H_* (\fa{\mop}X) [s^{-1}] \longrightarrow H_*(\fa{\mop } \ast \times \mC _\infty (X))[s^{-1}]
  \]
  is an isomorphism.  By the group completion theorem, 
  \[
   \g(\tau) \times \g( \pi)\colon \g( \fa\mop{X})\to  \g(\fa\mop\ast)\times  \g(\fa{\mC_{\infty}}X)
  \] 
  induces an isomorphism in homology. As both source and target are simple spaces, 
  the statement of the theorem will then follow by the Whitehead theorem.

  We have filtrations 
  \[ 
   F_n =\left(\coprod _{k\leq n} \op_{\infty}( k)\times _{\Sigma_k} X^k\right)/\sim
  \]
  of $\fa{\mop_{\infty}}X$ and 
  \[
   \tilde{F}_n =\left( \coprod_{k\leq n} \op_{\infty} (0)
   \times \cube_{\infty}(k) \times_{\Sigma_k}X^k \right)/\sim
  \]
  of $\fa{\mop_\infty}\ast\times \fa{\mC_{\infty}}X$.  
  The associated subquotients are 
  \[
   F_n/F_{n-1}=\op_{\infty}(n)\ltimes_{\Sigma_n}X^{\wedge n}
   \text{  and  }
   \tilde{F}_n/\tilde{F}_{n-1}= \op_{\infty}(0)\ltimes \cube_{\infty}(n) 
   \ltimes_{\Sigma_n}X^{\wedge n}
  \]
   where $\ltimes$ is the {\bf half smash product} defined by 
   $Y\ltimes Z\coloneqq (Y\times Z)/(Y\times \ast)$. 
  The map in \eqref{limitmap} is compatible with these filtrations, 
  and so there are induced maps 
  \begin{equation}\label{subquotients}
   F_n/F_{n-1}\to \tilde{F}_n/\tilde{F}_{n-1}.
  \end{equation} 
  There is a commutative diagram 
  \[\xymatrix{
    \op_{\infty}(n)\ar[r]\ar[d]_{\tilde{D}\times \pi}
    &\op_{\infty}(n)\times Y\ar[d]^{(\tilde{D}\times \pi)\times 1_Y}\\
    \op_\infty(0)\times \cube_{\infty}(n)\ar[r]&\op_\infty(0)\times(\cube_{\infty}(n)\times Y)
  }\]
  where the maps to $Y$ are inclusion of the base point.  
  Taking the quotient by the $\Sigma_n$ action (and observing that $\Sigma_n$ acts trivially on $\op_\infty(0)$),
  we have a map of cofiber sequences
  \[\xymatrix@C=12pt{
   \op_{\infty}(n)/\Sigma_n\ar[r]\ar[d]
   &\op_{\infty}(n)\times_{\Sigma_n}Y\ar[r]\ar[d]
   &\op_{\infty}(n)\ltimes_{\Sigma_n}Y\ar[d]\\
   \op_\infty(0)\times (\cube_{\infty}(n)\times_{\Sigma_n} \ast)\ar[r]
   &\op_\infty(0) \times(\cube_{\infty}(n)\times_{\Sigma_n}Y)\ar[r]
   &\op_\infty(0)\ltimes(\cube_{\infty}(n)\ltimes_{\Sigma_n}Y)
  .}\]
  The first two vertical maps are as in \eqref{splitmaps} and, 
  if we take $Y = X^{\wedge n},$ the last map is \eqref{subquotients}. 
  The long exact sequence in homology, \autoref{productstability}, 
  and the five lemma imply \eqref{subquotients} is an isomorphism in homology. 
  The filtrations above define a pair of spectral sequences 
  and \eqref{subquotients} defines a map between them.
  These spectral sequences converge to the homology of 
  $\fa{\mop_\infty}X$ and $\fa{\mop_\infty}\ast\times \fa{\mC_\infty}X$ and
  \eqref{limitmap} induces \eqref{subquotients}.  
  Since \eqref{subquotients} is an isomorphism in homology, 
  \cite[3.4]{mccleary} implies \eqref{limitmap} is an isomorphism in homology. 
\end{proof}


\section{OHS as infinite loop space operad}\label{sec:infinite_loop}

In this section  we use \autoref{split} 
to show that, for any OHS $\op$, the group completion of an $\op$-algebra $X$ is an infinite loop space. The idea is to replace $X$ by its free resolution as an $\op$-algebra using the bar construction. In \autoref{product}, we show that,  after group completing each level, the (homotopy) quotient of $\g X$ by $\g \op (*)$ is an infinite loop space. To see that $\g X$ is also an infinite loop space, in \autoref{freereplacement} we consider  $\op(*) \times X$ with the diagonal $\op$-algebra structure and deduce that the (homotopy)
quotient of $\g (\op(*) \times X) \simeq \g (\op(*)) \times \g (X)$
by $\g (\op (*))$ is $\g (X)$ and hence an infinite loop space. Our main theorem, \autoref{loopsandaction}, summarizes these results.  

We start by proving the naturality of our construction. 

\begin{prop}
  \label{infiniteloop} 
  Suppose that $\op$ is an OHS with compatible operad map  $\pi\colon\op \to \cube_\infty$.
  The assignment $X\mapsto |\g B_\bullet(\mC_\infty, \mop, X)|$ defines 
  a functor from $\op$-\oalgs{} to $\Omega^\infty$-spaces.
\end{prop} 

\begin{proof} There is a map of monads $\alpha\colon \mC_{\infty} \to \Omega^\infty \Sigma^\infty$
 \cite[5.2]{May} and, for every based space $Z$,  the map 
 $\alpha\colon \mC_{\infty}Z \to \Omega^\infty \Sigma^\infty Z$ is a group completion \cite[2.2]{May2}. 

  For a map $f\colon X\to Y$ of $\op$-\oalgs{}, we have a
  commutative diagram where all horizontal maps are induced by $f$.
  \[\xymatrix{
   |\g B_\bullet(\mC_\infty, \mop, X)|\ar[d]\ar[r]
   &|\g B_\bullet( \mC_\infty, \mop,  Y)| \ar[d]
   \\
   |\g B_\bullet(\Omega^\infty \Sigma ^\infty, \mop,  X)|\ar[r]
   &|\g B_\bullet( \Omega^\infty \Sigma ^\infty, \mop,  Y)|
   \\
      |\g \Omega^\infty B_\bullet(\Sigma ^\infty, \mop,  X)|\ar[r]\ar[u]
   &|\g  \Omega^\infty B_\bullet(\Sigma ^\infty, \mop,  Y)|\ar[u]
   \\
    |\Omega^\infty B_\bullet(\Sigma ^\infty, \mop, X)|\ar[d]\ar[u]\ar[r]
   &|\Omega^\infty B_\bullet( \Sigma ^\infty, \mop, Y)|\ar[d]\ar[u]
   \\
   \Omega ^\infty |B_\bullet(\Sigma ^\infty, \mop, X)|\ar[r]
   &\Omega ^\infty |B_\bullet(\Sigma ^\infty, \mop, Y)|
   } \]
  The first vertical maps are induced by  $\alpha \colon \mC_\infty \to \Omega^\infty \Sigma^\infty$ and 
  are levelwise weak equivalences since they are levelwise homology isomorphisms between simple spaces.
  The second maps are  equivalences by \autoref{bar_construction_properties}.  Since infinite loop spaces
  are grouplike, the third maps are equivalences by 
  \autoref{gcthrm}.   
  The fourth maps are equivalences by 
  \cite[12.3]{May}. 

  Note that the last horizontal map in the diagram above is a map  of $\Omega^\infty$-spaces.
\end{proof}

The inclusion of the base point $\ast \to X$ induces a map of $\mop$-\malgs{} $\fa\mop\ast \to \fa{\mop}X$ 
and a simplicial map $\g \fa\mop\ast \to \g B_{\bullet}(\mop,\mop,X)$, where  $\g \fa\mop\ast$ is a constant simplicial space.

The operad maps $\cA \to  \op \xto{\pi} \cube_\infty$ induce  maps of monads 
$\mA \to \mop \xto{\pi} \mC_\infty$ and a map of simplicial $\mA$-\malgs{} 
$B_\bullet (\mop, \mop, X)\to B_\bullet (\mC_\infty , \mop, X).$  
Applying the functor $\g$ levelwise, we have a simplicial map 
\[
  \g (B(\pi,1,1)) \colon \g B_\bullet (\mop, \mop, X)
  \to \g B_\bullet (\mC_\infty , \mop, X).
\]

\begin{lem}
  \label{product} 
  Let $\op$ be an OHS with a compatible operad map  
  $\pi\colon \op \to \cube_{\infty}.$  For any  $\op$-\oalg{} $X$ 
   there  is a homotopy fibration sequence
  \[ 
    \g\fa\mop\ast \to |\g B_\bullet (\mop, \mop, X)|
    \xto{\g B(\pi, 1, 1)} |\g B_\bullet (\mC_\infty , \mop, X)|. 
  \]
\end{lem}
  
\begin{proof} For each fixed $n$ we have a commutative  diagram 
  \[\xymatrix{
   \g \fa\mop\ast\ar@{=}[d]\ar[r]&\g \mop^{n+1}X \ar[r]\ar[d]_{\simeq}
   & \g \mC_\infty \mop^n X \ar@{=}[d]
   \\
   \g\fa\mop\ast\ar[r]^-\theta
   &\g \mop(\ast) \times \g \mC_\infty\mop^n X \ar[r]^-{\proj} 
   & \g \mC_\infty\mop^n X 
  }\]
  where the map $\theta$ is the  identity on the first factor and the composite 
  \[
    \g \fa\mop\ast \to \g \mop (\mop^n X) \xto{\pi} \g\mC_\infty \mop^n X
  \] 
  on the second factor.  
  The center vertical arrow is the homotopy equivalence of \autoref{split}. 
  The bottom row is a fiber sequence up to homotopy, and so is the top row.  
  
  In order to deduce that the top row is a homotopy fibration sequence after realization, we need to show that the conditions of \cite [B.4] {bousfieldfriedlander} are satisfied, namely, that the simplical spaces $\mathcal{G} B_\bullet(\mop, \mop, X)$ and $\mathcal{G} B_\bullet(\mC_\infty, \mop, X)$ satisfy the $\pi_\ast$-Kan condition and that the map between them is a fibration on $\pi_0$.
  This follows since  both the total and base  simplicial spaces are levelwise loop spaces,  and the structure maps for both simplicial spaces as well as the map between them are maps of loop spaces. 
  To be more  explicit, let $Z_\bullet$  be a simplicial space with $Z_m$ simple for  $m\geq 0$, and let $\pi_t (Z_m)_{free} $ denote the set of unpointed homotopy classes of maps from the $t$-sphere to $Z_m$. Then $Z_\bullet$ satisfies the $\pi_*$-Kan condition if and only if  the simplicial map $\pi_t (Z_\bullet)_{free} \to \pi_0 (Z_\bullet)$ is a fibration for each $t\geq 1$ \cite[B.3.1] {bousfieldfriedlander}. Thus both $\mathcal{G}B_\bullet(\mop,\mop,X)$ and 
  $\mathcal{G}B_\bullet(\mC_\infty, \mop, X)$ satisfy the $\pi_*$-Kan condition, and the induced map on connected components
  \[
    \pi_0 \mathcal{G}B_\bullet (\mop, \mop,X) \to \pi_0 \mathcal{G}B_ \bullet (\mC_\infty, \mop, X)
  \]
  is a map of simplicial groups and hence a fibration.
\end{proof}

If $X$ is an $\op$-algebra, there is a map of  $\mop$-\malgs{} 
$\fa{\mop}{\ast} \to  \fa{\mop}{X} \xto{\xi} X$ and, 
applying group completion, this defines a map $\g\mop(\ast)\to \g X$. 
An operad map $\cA\to \op$ defines a monoid structure on $X$ so  
the composite 
\begin{equation}\label{action} 
  \g\fa{\mop}{\ast} \times \g X \to \g X \times \g X \to \g X,
\end{equation}  
where the second map is loop sum, is an action $\g\fa\mop\ast$ on $\g X$ up to homotopy.
\vskip .2in

\begin{prop}\label{freereplacement}
  Let $\op$ be an operad with homological 
  stability and suppose 
  there is a compatible  map  $\pi\colon \op \to \cube_{\infty}$. 
  For any $\op$-\oalg{} $X$, 
  there is a weak homotopy equivalence between  
   $ \g X$ and  $|\g B_\bullet( \mC_\infty,\mop, \fa\mop\ast\times X)|$
  where $\fa\mop\ast\times X$ has the product $\op$-\oalg{} structure.
\end{prop}

\begin{proof}
  \autoref{product}, applied to the product $\op$-\oalg{} $\fa\mop\ast\times X$,  
  implies the top horizontal row in the following diagram is a fibration sequence:
  \[\xymatrix{
    \cG \fa\mop\ast \ar@{=}[dd] \ar[r]^-{\iota} &  
    |\cG B_\bullet(\mop, \mop, \fa\mop\ast \times X)| \ar[d]^{\g(\eval)} \ar[r] &
    |\cG B_\bullet (\mC_\infty, \mop, \fa\mop\ast\times  X)| \\
    &\cG (\fa\mop\ast \times X) \ar[d]^\simeq
    \\
     \cG \fa\mop\ast \ar[r]^-{\triangle} & \cG\fa\mop\ast \times \cG X &
   }\]
   Let $x_0$ be the base point of $X$.
  The map $\iota$ is obtained by mapping $*$ to  the basepoint 
  $((1,\ast), x_0) \in \fa\mop\ast \times X$ and  extending freely to a map 
  $\fa\mop\ast \to \fa\mop{(\fa\mop\ast \times X)}$.   The map  $\eval$ is induced by 
  the diagonal action of the monad $\mop$ on $\fa\mop\ast \times X$.  
  For the diagram above to commute, the image of $a \in \cG (\fa\mop\ast)$
  under $\triangle$ must be $(a, a\bar{x}_0) \in \cG \fa\mop\ast \times \cG X $ 
  where $\bar{x}_0$ denotes the image of $x_0$ in $\g X$ and $a\bar{x}_0$ 
  is the image of $(a,\bar{x}_0)$ under  \eqref{action}.  
  The vertical maps are homotopy equivalences, and so there is a fibration sequence 
  \[\xymatrix{
    \cG \fa\mop\ast \ar[r]^-{\triangle} & \cG \fa\mop\ast \times \cG X \ar[r]
    &|\cG B_\bullet (\mC_\infty, \mop,  \fa\mop\ast \times  X)|.
  }\]
  
  Now consider the {\bf shearing} map 
  $\sh\colon \cG \fa\mop\ast  \times \cG X \to \cG \fa\mop\ast \times \cG X$
  defined by $\sh ((a, x)) = (a, ax)$. As $\cG \fa\mop\ast$ is a group up to homotopy, 
  $\sh$ is a homotopy equivalence with homotopy inverse given by $\sh^{-1}((a,x)) = (a, a^{-1} x)$. 
  There is a commutative diagram
  \[\xymatrix{
    \cG \fa\mop\ast \ar[d]^= \ar[r] ^-{\inc_1} & \cG \fa\mop\ast \times \cG X \ar[d]^{\sh}\ar[r] & \cG X \\
    \cG \fa\mop\ast \ar[r]^-{\triangle} &\cG \fa\mop\ast \times \cG X \ar[r] & 
    |\cG B_\bullet (\mC_\infty , \mop, \fa\mop\ast  \times X)|
  }\] 
  where both rows are fibration sequences.
  The homotopy fiber of $\inc_1$ can be identified with
  $\Omega \cG X$.  The homotopy fiber of  $\triangle$ can be identified with
  $\Omega |\cG B_\bullet (\mC_\infty , \mop, \fa\mop\ast \times X)|$.
  As both vertical arrows are homotopy equivalences, so is the induced map between the 
  homotopy fibers. Since the above commutative diagram is a diagram of loop spaces, we 
  conclude that $\cG X$ and $|\cG B_\bullet (\mC_\infty , \mop, \fa\mop\ast \times X)|$ are homotopy equivalent. 
\end{proof}

We can now state and prove our main theorem.

\begin{thm}
  \label{loopsandaction}  Suppose $\op$ is an OHS. Then group completion   
  defines a functor from $\op$-\oalgs{} to $\Omega^\infty$-spaces. In particular,
  $\g \fa\mop\ast$ is equivalent to an $\Omega ^\infty$-space and, for every $\op$-algebra $X$, there is an 
  $\Omega^\infty$-map 
  \[
    \g \fa\mop\ast\times \g X\to \g X,
  \]
  compatible with $\op$-algebra maps,
  where the source is given the product $\Omega ^\infty$-space structure.
\end{thm}

\begin{proof}
  By \autoref{cor:productohs} we may assume without loss of generality that $\op$ comes equipped with a map 
  $\pi \colon\op \to \cube _\infty$
  which is compatible with the multiplication $\mu\colon \cA \to \op$. Hence, we may apply \autoref{infiniteloop}, to conclude that 
  $|\g B_\bullet(\mC_\infty, {\mop}, \fa{\mop}\ast\times X)|$ 
  is an $\Omega^\infty$-space.  By \autoref{freereplacement},
  \[
    \g X\simeq |\g B_\bullet( \mC_\infty, {\mop}, \fa{\mop}\ast\times X)|.
  \] 
  For a map of $\op$-\oalgs{} $f\colon X\to Y$ we have an infinite loop map    
  \[\xymatrix{
    |\g B_\bullet( \mC_\infty, {\mop}, \fa{{\mop}}\ast\times X)|\ar[r]&
    |\g B_\bullet(\mC_\infty, {\mop}, \fa{{\mop}}\ast\times Y)|.
  }\]  
  
  The action of $\g\fa\mop\ast$ on $\g X$ is as defined in \eqref{action}. 
  The maps $\fa\mop\ast\to \fa\mop{X}\to X$ are maps of $\op$-\oalgs{} and so $\g \fa\mop\ast
  \to \g X$ is a map of $\Omega^\infty$-spaces.
  In \eqref{action} 
  the first map is an infinite loop map and the second is loop sum.  Since  loop 
  sum is an $\Omega^{\infty}$-map
  for $\Omega^\infty$-spaces, the result follows. 
\end{proof}

We  remark that, via the map $\pi\colon \op \to \cube_\infty $, any infinite loop space is a group-complete $\op$-\oalg{}. However,  the induced action by $\g \fa\mop\ast$  factors through $\pi\colon \g\fa\mop \ast \to \g \mC_\infty(\ast) \simeq \Omega ^\infty S^\infty(\ast)\simeq \ast$, and is thus trivial. 

\begin{eg}[Abelian monoids]
  A topological abelian monoid $A$  defines an operad $\op$ where 
  \[
    \op(n) \coloneqq A
  \]
  with the  trivial $\Sigma_n$ action and $\ast = 1 = 0_A.$ 
  The structure maps in $\op$ are given by the monoid multiplication
 and  we  consider $\op$ a graded operad concentrated in degree $0$.  There is a canonical map of operads $\mu\colon \As \to \op$ which factors through $\Com$ and sends
  $\As (n)$ to the identity element of $\op (n)= A$. Thus $\mu $ satisfies  condition  \ref{htpycom}.

For this operad $D\colon \op (n) \to \op(0)$ is the identity map on $A$, 
and so using \autoref{loopsandaction} every group-like $\op$-\oalg{} is an $\Omega^{\infty}$-space and 
  comes equipped with an $\Omega^{\infty}$-space action by the group completion of $A$.

  We can characterize $\op$-algebras as abelian monoids $X$ with a monoid map $f\colon A\to X$.  
 Then the consequences of \autoref{loopsandaction}  are familiar for this operad
 since  any abelian monoid  $X$ is an $E_\infty$-algebra, and the group completion $\Omega B X$ is an infinite loop space.  
\end{eg}

\section{Higher dimensional manifolds}\label{highdim}
In \cite{GRWII} Galatius and Randal-Williams prove a general homological stability result for even dimensional manifolds. We utilize their result to provide new examples of operads with homological stability. 
 We first recall the set-up from \cite {GRWII}.  

\subsection{Homological stability for spaces of manifolds}
Let $W$ be a $2k$-dimensional smooth, compact manifold with boundary $P$. Let 
$\theta \colon B\to BO(2k)$ be a fibration and $\theta^* (\gamma_{2k})$ be the pull-back of the canonical $2k$-bundle over $BO(2k)$. 
Fix a bundle map  
\[ 
 \ell_P \colon \epsilon^1 \oplus TP \to \theta^*\gamma_{2k}
\] 
and let $\text{Bun}_{k, \partial}^{\theta}(TW; \ell_P)$ 
denote the space of all bundle maps $\ell\colon TW\to\theta^*\gamma_{2k}$ where the restriction of $\ell$ to $P$ is $\ell_P$  and the  underlying map $\ell \colon W\to B$ is $k$-connected. 
We call a bundle map $\ell_W\colon TW\to \theta^*\gamma_{2k}$ a $\theta$-\textit{structure} on $W$ and define a moduli space associated to $W$ and a given $\theta$-structure $\ell_W$ by letting
\[
  \mathcal{M}_k^{\theta}(W; \ell_W ) \subset \text{Bun}_{k,\partial}^\theta (TW; \ell_P)//\Diff(W; \partial)
\]
be the connected component determined by $\ell_W$. 
Here $\Diff(W; \partial)$ is the topological group of diffeomorphisms of $W$ that restrict to the identity on a neighborhood of $\partial W$ and $//$ denotes the homotopy orbit space (also called the Borel construction). We refer the reader to \cite{GRWII} for more details.

To construct maps between moduli spaces, let $M$ be a cobordism between $P$ and $Q$ equipped with a $\theta$-structure 
$\ell_M$ which restricts to $\ell_P$ over $P$ and to $\ell_Q$ over $Q$. Then there is a map
\[
  \cup_P(M, \ell_M)\colon \  \mathcal{M}_k^{\theta}(W; \ell_W) \to \mathcal{M}_k^{\theta}(W';\ell_{W'}) 
\]
where $W' = W \cup _P M$ and $\ell_{W'} = \ell_W \cup \ell _M$.
\begin{thm} \cite[1.7]{GRWII}\label{thm:grw}
If $2k \geq 6$, 
$B$ is simply connected, and $(M, P)$ is $(k-1)$-connected, the  map $\cup_P(M, \ell_M)$
induces an isomorphism on homology with constant coefficients in degrees 
$* \leq (g-4)/3$.
\end{thm}

A $g$-fold connected sum 
of $S^k \times S^k$ from which $n$ disks have been removed is a manifold of type $W_{g,n}$.
Then the {\bf genus} $g$ of $(W,\ell_W)$ is  the maximal number of disjoint
copies of $W_{1,1}$  in $W$ where each $W_{1,1}$ has a $\theta$-structure that is admissible in the sense of \cite[1.2]{GRWII}. 

\begin{rmk}[Addendum to \autoref{thm:grw}]\label{thm:grw:add}
The above statement is a finite genus homological stability statement. However, the crucial property for an OHS only requires an infinite genus homological stability property! This is proved in \cite[1.3]{GRWII} {\bf for all $2k \geq 2$}. More concretely,
the map $\cup_P(M, \ell_M)$ induces isomorphisms in homology in any even dimension with any abelian system of local coefficients from the homotopy limit
of $\mathcal{M}_k^{\theta}(W; \ell_W)$ under gluing operation  $\cup_P (H_P,\ell_{H_P})$ to the homotopy colimit of $\mathcal{M}_k^{\theta}(W';\ell_{W'})$ under the gluing operation 
$\cup_Q (H_Q,\ell_{H_Q})$.
Here $H_P$ is the connected sum of $S^k\times S^k$ with $[0,1] \times P$, and the manifold $H_Q$ is defined similarly. Again, the $\theta$-structures $\ell_{H_P}$ and $\ell_{H_Q}$ have to be admissible in the sense of \cite[1.2]{GRWII}. 
\end{rmk}

\subsection{A manifold operad}\label{sec:mfld_operad}
The cobordism categories of embedded manifolds in Euclidean space in \cite {GMTW} and subsequent work by Galatius and Randal-Williams give rise to $\Gamma$-spaces in the sense of \cite {segalcats} but they are not symmetric monoidal. Therefore we cannot simply use an appropriate subcategory to define our operad. Instead we follow  the approach for surface categories  using atomic cobordisms first introduced in \cite {Tillmann99}.

Fix a fibration $\theta\colon  B\to BO(2k)$ where $B$ is $k$-connected. For example, take $B$ to be the $k$-connective cover of $BO(2k)$.
Let $Q = S^{2k-1}$ and consider fixed atomic $2k$-cobordisms 
 \begin{itemize} 
  \item a handle $H\colon Q \mapsto Q$,  
  \item a sphere with three holes $M\colon Q \sqcup Q \mapsto Q$,  
  \item a disk $D\colon \emptyset \mapsto Q$, 
 \end{itemize}
 where $H$ is of type 
 $W_{1, 1+1}$, $M $ of type $W_{0,2+1}$ and $D$ of type $W_{0,1}$, all with parametrized boundaries.
 See \autoref{atomic_surfaces} for atomic cobordism for  $ k = 1$ and $B = BSO(2)$.
 Let $\ell_Q$ be a $\theta$-structure on $Q$ such that $\ell_Q \sqcup \ell_Q$ extends to a $\theta$-structure 
 $\ell_H$ on $H$, $\ell_Q \sqcup \ell_Q \sqcup \ell_Q$ extends to a $\theta$-structure
 $\ell_M$ on $M$, and $\ell_Q$ extends to a $\theta$-structure $\ell _D$ on $D$. We will call the quadruple $\{\ell_Q, \ell_H, \ell_M, \ell_D\}$ an {\bf operadic} $\theta$-structure. 
Gluing the atomic pieces along their boundaries (as cobordisms) produces new cobordisms $W$ of type 
\[W_{g, n+1}\colon \coprod _n Q \mapsto Q\]
each equipped with a $\theta$-structure $\ell_W$. 

\begin{figure}
\begin{tikzpicture}
  \draw [ultra thick] (0,1) ellipse (.10 and .5);
  \draw (0,.5) to [out=0,in=-90] (.7,1)  to [out=90,in=0]  (0,1.5);
  \node at (.25, -.5){$D$};
  
  \draw [ultra thick] (1.8,1) ellipse (.10 and .5);
  \draw [ultra thick] (4.2,1) ellipse (.10 and .5);
  \draw (1.8,.5) to [out=-5,in=180] (3,0)  to [out=0,in=185]  (4.2,.5);
  \draw (1.8,1.5) to [out=5,in=180] (3,2)  to [out=0,in=175]  (4.2,1.5);
    \draw (2.5,1.1) to [out=-45,in=180] (3,.85)  to [out=0,in=225]  (3.5,1.1);
  \draw (2.7,.9) to [out=45,in=180] (3,1.05)  to [out=0,in=135]  (3.3,.9);
  \node at (3, -.5){$H$};

  \draw [ultra thick] (5.3,1) ellipse (.10 and .5);
  \draw [ultra thick] (7.5,0) ellipse (.10 and .5);
  \draw [ultra thick] (7.5,2) ellipse (.10 and .5);
  \draw (5.3,1.5) to [out=5,in=180]  (7.5,2.5);
  \draw (5.3,.5) to [out=-5,in=180]  (7.5,-.5);
  \draw (7.5,1.5) to [out=200,in=80] (6.75,1)  to [out=-80,in=150]  (7.5,.5);
    \node at (6.3, -.5){$M$};
\end{tikzpicture}
\caption{Orientable Atomic Surfaces}\label{atomic_surfaces}
\end{figure}
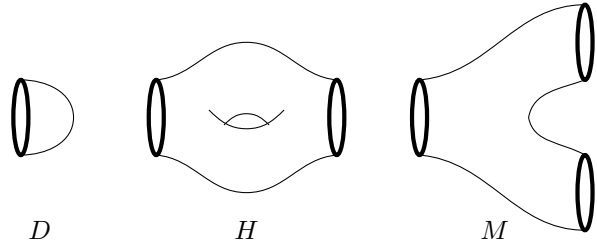
  
For $(g, n)$, let $\cC_{g,n}$ be the groupoid whose 
objects are pairs $(W, \ell)$ where 
\begin{itemize}
  \item $W$  is a cobordism of type $W_{g, n+1}$  constructed from the atomic manifolds and equipped with an ordering of its $n$ incoming boundary components, and
  \item $\ell \colon \, TW \to \theta^*(\gamma_{2k})$ is a $\theta$-structure  in the path-connected component  of $\ell_W$ in the space of all  bundle maps $TW \to \theta^*(\gamma_{2k})$ or in the orbit of such a map under the action of the group $\Diff (W; \partial)$ of diffeomorphisms of $W$ that fix a collar of the boundary. 
\end{itemize}
By assumption $B$ is $k$-connected and $W$ is $(k-1)$-connected and hence $\ell$ is $k$-connected.
A  morphism between $(W, \ell)$ and $(W', \ell ')$ is a diffeomorphism
$\phi \colon W \to W'$ that commutes with the parametrizations of the boundary components (and collars) and the ordering of the incoming boundary components. Furthermore,
we require that 
\[
  \ell = D\phi \circ \ell'\colon \,  TW \to \theta^*(\gamma_{2k}).
\]

\begin{lem}\label{lem:theta_classifying}
  There is a homotopy equivalence 
  \[
   B\cC _{g,n} \simeq \cM ^\theta _k (W_{g,n+1}, \ell_{W_{g,n+1}}).
  \]
\end{lem}

\begin{proof} 
  For $W\in \cC_{g,n}$, let $X_W$ denote  the space of all  $\theta$-structures $\ell$ which are  in the path-connected component  of $\ell_W$ in the space of all  bundle maps $TW \to \theta^*(\gamma_{2k})$ or in the orbit of such a map under the action of the group $\Diff (W; \partial)$.
  Then, for any other $W' \in \cC_{g,n}$, $X_W$ is homotopy equivalent to $X_{W'}$, and given a diffeomorphism $\phi\colon W \to W'$, the $\theta $-structure $D\phi \circ \ell'$ is in $X_W$ whenever $\ell' \in X_{W'}$.  
  Thus by construction $\cC_{g,n}$ is equivalent as a topological category to its full
  subcategory on objects of the form $(W, \ell)$ for some fixed $W$. This subcategory is the translation category of $X_W$ 
  under the action of $\Diff (W; \partial)$. Then the nerve of this subcategory has the homotopy type of $X_W // \Diff (W; \partial)$. By definition this is $\cM ^\theta _k (W, \ell_W)$.
\end{proof}

We then form a manifold operad $\cW^{2k}$ by defining
 \[\cW ^{2k}(n)=\coprod_{g\geq 0 } B\cC _{g,n}.\]
The operad is graded by the genus   and the structure maps $\gamma$ are defined by gluing outgoing boundary spheres of one cobordism to incoming boundary spheres of others. 
In the next subsection we will identify an $A_\infty$-operad as a sub-operad of the genus zero sub-operad $\cW^{2k}_0$ of $\cW^{2k}$ so that
condition \ref{htpycom} is satisfied.   

Since the pair  $(D, Q) = (D^{2k}, S^{2k-1})$ is $(k-1)$-connected, for $2k\geq 2$ the results of Galatius and Randal-Williams quoted above, \autoref{thm:grw} and \autoref{thm:grw:add}, show that
\[
 D\colon \cW _\infty (n) \longrightarrow \cW _\infty (0)
\]
is a homology isomorphism.
Thus we have proved the following result.

\begin{thm}\label{thm:even_mfld_stability}
For $2k\geq 2$, 
the operad $\cW^{2k}$ associated to a fibration $\theta \colon \, B \to BO(2k)$ with $B$ $k$-connected and an operadic choice of 
  $\theta$-structures defines an operad with homological stability.
\end{thm} 

\begin{eg}\label{eg:BOk}
  Let $\theta\colon BO(2k) \left< k \right> \to BO(2k)$ be the $k$-connected cover of $BO(2k)$.  Assume that our atomic cobordisms $D,M,$ and $H$ are embedded in $[0,1] \times \mathbb R^\infty$ as cobordisms, that is, the outgoing boundary $\partial _+$ is contained in $\{0\} \times \mathbb R^\infty$ and the incoming boundary $\partial _-$ is contained in $\{1 \}\times \mathbb R^\infty$; each boundary component is equipped with a collar
   $\partial_+ \times [0, \epsilon)$, respectively $\partial _-\times (\epsilon , 1]$; and the boundary spheres can be mapped to each other by rigid translations in $\mathbb R \times \mathbb R^\infty$. Then the Gauss map defines a map from the atomic cobordisms (and hence any $W$ built from them) to $BO(2k)$ which on (and near) the boundary restricts to a map $\ell_Q\colon \epsilon^1 \oplus TQ \to \gamma _{2k}$.
  As each atomic cobordism and each $W$ built from them is $k$ connected, the Gauss maps have compatible lifts to $BO(2k)\left< k \right>$, their Moore-Postnikov $k$-stage. Thus, for this $\theta$-structure, $\cW^{2k} \simeq B\Diff (W_{g, n+1};\partial)$.
\end{eg}

As an immediate consequence of  \autoref{thm:even_mfld_stability}, we have that
\[
  \Omega B \left(\coprod _{g\geq 0} \cM ^\theta _k (W_{g, 1}, \ell_{W_{g,1}})\right) = \Omega B \cW ^{2k} (0)
\]
is homotopy equivalent to an  infinite loop space.
Furthermore, on application of the group completion theorem, we deduce that it is homotopy equivalent to  
\[
  \mathbb Z \times \left(\hocolim \limits_{g \to  \infty} \cM ^\theta _k (W_{g,1}, \ell_{W_{g,1}})\right)^+ .
\]
After restricting to a connected component, the analog of the Madsen-Weiss Theorem for higher dimensional manifolds and 
\cite[1.8] {GRWII}, imply 
\begin{equation}\label{mtidentification}
  \Omega B_0 \left(\coprod _{g\geq 0} \cM ^\theta _k (W_{g, 1}, \ell_{W_{g,1}})\right) \simeq 
  \left(\hocolim \limits_{g \to  \infty} \cM ^\theta _k (W_{g,1}, \ell_{W_{g,1}})\right)^+ \simeq
  \Omega ^\infty _0 {\bf  MT} \theta.
\end{equation}

\begin{rmk}
  Maps $\cM ^\theta _k (W_{g,1+1}, \ell_{W_{g,1+1}}) \to \Omega ^\infty {\bf MT} \theta$ which are compatible with gluing (a strictly monoidal structure on the disjoint union over all $g$) were constructed in \cite {GMTW}.
  Following ideas in \cite {Til97}, a map of strict monoids can be constructed
  \[
    \coprod _{g\geq 0} \cM ^\theta _k (W_{g, 1}, \ell_{W_{g,1}}) \longrightarrow 
    \coprod _{g\geq 0} \cM ^\theta _k (W_{g, 1+1}, \ell_{W_{g,1+1}})
  \]
  which on composition and group completion induces the homotopy equivalence on a connected component in \eqref{mtidentification}.
  We  do not know whether the evident infinite loop space structure on
  $\Omega ^\infty_0 {\bf MT} \theta$ is compatible under this map with the infinite loop space structure produced by our methods on 
  the left hand term in \eqref{mtidentification}  though we note that for the case $2k=2$ this was affirmed by Wahl in \cite {Wahl04}.
\end{rmk}

\subsection{Manifold models of \texorpdfstring{$\protect A_\infty$}{A}-operads}
For each dimension $d\geq 2$ we will construct an $A_\infty$-operad $\cA^d$ build from $d$-dimensional spheres and diffeomorphisms.

Let $S\colon I \sqcup I \to I$ and $R\colon \emptyset \to I$ be  fixed cobordisms from two copies of the unit interval $I$ and the empty set, respectively, to $I$.  See \autoref{fig:tiles}.
By gluing  copies of $S$ and $R$ as cobordisms along incoming and outgoing boundary intervals, and labeling their 
incoming boundary components, we can construct cobordisms $T$
from $n$ copies of $I$ to $I$. We denote the union of the $n+1$ copies of $I$ in the boundary by $\partial_0$. 

\begin{figure}
\begin{tikzpicture}
  \draw [ultra thick] (0,.5) to (0, 1.5);
  \draw (0,.5) to [out=0,in=-90] (.7,1)  to [out=90,in=0]  (0,1.5);
   \node at (.25, -.5){$R$};

  \draw [ultra thick] (2.3,.5) to (2.3, 1.5) ;
  \draw [ultra thick] (4.5,-.5) to (4.5, .5);
  \draw [ultra thick] (4.5,1.5) to (4.5, 2.5);
  \draw (2.3,1.5) to [out=5,in=180]  (4.5,2.5);
  \draw (2.3,.5) to [out=-5,in=180]  (4.5,-.5);
  \draw (4.5,1.5) to [out=200,in=80] (3.75,1)  to [out=-80,in=150]  (4.5,.5);
  \node at (3, -.5){$S$};
\end{tikzpicture}
\caption{Tiles}\label{fig:tiles}
\end{figure}

Let $\cT_n$ be the groupoid whose objects are all such cobordisms and whose morphisms from $T$ to $\tilde T$ are the diffeomorphisms
from $T$ to $\tilde T$ that commute with the parametrizations of the $n$ labeled
incoming and the outgoing copies of $I$:
\[
\cT_n (T, \tilde T) = \Diff (T, \tilde T; \partial_0).
\]
We note that each of these morphism spaces is homotopy equivalent to the space of diffeomorphisms of the two-dimensional disk $D^2$ that fix the boundary, and hence is contractible. The symmetric group $\Sigma_n$ acts on $\cT _n$ by relabeling the incoming boundary components.
Gluing of incoming to outgoing copies of $I$ makes  $\cT = \{ \cT_n\}_{n\geq 0}$  an operad in categories (enriched in spaces), and its classifying space $B\cT = \{ B\cT_n\}_{n\geq 0}$ an operad 
in spaces. Because the morphism spaces $\cT_n (T, \tilde T)$ are contractible, so is
$B\cT_n$. We summarize with the following lemma.

\begin{lem}
$B\cT$ is an $A_\infty$-operad.
\qed 
\end{lem}

Next, for $d\geq 2$, consider for each tile $T \in \cT_n$ its thickening,
the space 
\[
T^d=T \times I^{d-1}.
\]
As $T$ is homeomorphic to the disk $D^2$,
$T^d$ is homeomorphic to the $d+1$ dimensional disk, and its boundary 
$\partial T^d$ is
homeomorphic to the sphere $S^d$. Furthermore, 
\[
\partial T^d = \partial T \times I^{d-1} \cup T \times \partial I^{d-1}
\] 
comes with $n$ incoming and one outgoing cubes $I^d$ given by
$\partial _0 \times I^{d-1}$. Define a new 
groupoid $\cT ^d_n$ with objects 
$\partial T^d$ for each tile $T \in \cT _n$, 
and morphisms $\cT^d_n (\partial T^d, \partial \tilde T^d)$ to be all
diffeomorphisms that are the restrictions to the boundary of diffeomorphims from $T^d$ to $\tilde T^d$ that have the form $\phi \times \id_{I^{d-1}}$ for some diffeomorphism of tiles   $\phi \in 
\cT (T, \tilde T)$.
Taking connected sum along the fixed cubes $I^d$ defines the structure maps of the operad.
By construction, $\cT_n$ is isomorphic to $\cT^d_n$ as a category enriched in spaces. 
Define 
\[\cA^d \coloneqq B\cT^d = \{ B\cT^d_n\}_{n\geq 0}. 
\]

\begin{cor}
  $\cA^d $ is an $A_\infty$-operad. \qed
\end{cor}

Now let $d=2k$ and identify the thickening of the atomic tiles $S$ and $R$ with the atomic $2k$-cobordisms $M$ and $D$.
We want to extend this to a map of operads
\[
  \mu\colon \cA^{2k} \longrightarrow \cW^{2k}.
\]
This can easily be done when the $\theta$-structure is  as in \autoref{eg:BOk}. 
For more general $\theta$-structures the operad $\cA^{2k}$ needs to be adapted as follows.
If $\ell_Q$ extends to $\ell_M$ and $\ell_D$ on $M$ and $D$ respectively, then each cobordism $W$ built from these will have an operadic  $\theta$-structure glued together from these, and hence each of our objects $\partial T^d$ has such an operadic $\theta$-structure.   
We thus need to thicken the object space  in $\cA^d$. For each tile $T$ we have   a family of objects  $(\partial T ^d , \ell)$ where $\ell$ is the pullback of the operadic $\theta$-structure  along a  diffeomorphisms of $\partial T^d$ in 
$\cA ^d$. But as the space of all such diffeomorphisms is contractible, the homotopy type  of the operad is not affected, and in particular the resulting operad is still an $A_\infty$-operad.
Finally we note that as $ \cW^{2k} (2)$ is path-connected, $\mu (\cA^{2k}(2)) $ is contained in one path-component and hence $\mu$ satisfies the homotopy commutativity  condition \ref{htpycom}.

\begin{rmk}Strictly speaking, the thickening $T^d$ is not smooth and 
  care needs to be taken when identifying $\partial T^d$ with $S^d$.  In particular, we want
the maps $\phi \times \id_{I^d}$ to  produce diffeomorphisms on $S^d$. This can easily be achieved by considering only those diffeomorphisms from $T$ to $\tilde T$ that restrict
on a collar $\partial _f \times [0, \epsilon)$ 
of the free boundary $\partial _f =\partial T \setminus \partial_0$
to diffeomorphisms of the form $\alpha \times \id_{[0, \epsilon)}$ for some 
smooth $\alpha \colon \partial _f (T ) \to \partial _f (\tilde T)$.  
We leave the details to the dedicated reader.
\end{rmk}

\subsection{Diffeomorphism groups for other tangential structures.}
We will show that the diffeomorphism  and  mapping class groups of $W_{g,1}$ also give rise to infinite loop spaces.
We first construct a $\cW^{2k}$-algebra consisting of classifying spaces of diffeomorphisms groups. 

We define a category $\cD _g$ that is entirely analogous to $\cC_{g,0}$ in \S\ref{sec:mfld_operad}. 
For each $W$ built out of atomic $2k$-cobordisms of type $W_{g,1}$ the categeory  $\cD_g$ has  one object, and its morphisms are diffeomorphisms that fix the parametrization of the boundary. 
Let $F\colon\cC_{g,0} \to \cD_g$ be the functor that forgets the $\theta$-structure on objects. By the definition of $\cD_g$ this map is surjective. As $\cD_g$ is a groupoid, its classifying space is 
\[
B\cD_g \simeq B\Diff(W_{g,1}; \partial).
\]
The forgetful functor induces a map 
\[
BF\colon B\cC_{g,0} \to B \cD_g
\]
which, under the above homotopy equivalence and \autoref{lem:theta_classifying}, corresponds naturally to the
bundle map 
\[
X_{W_{g,1}} // \Diff (W_{g,1}; \partial) \to B\Diff (W_{g,1}; \partial).
\]
As $F$ commutes with gluing, it gives $\coprod_{g\geq 0}B\cD_g$
the structure of an algebra over  $\cW^{2k}$. 

\begin{cor}
The space $\mathbb Z \times  (\hocolim _{g\to \infty}B\Diff (W_{g,1})) ^+$
is an infinite loop space.
\end{cor}

\begin{proof}
Pick an operadic $\theta$-structure on $Q = S^{2k-1}$, as in \autoref{eg:BOk},
and consider the associated operad $\cW^{2k}$. We have argued above that 
\[
\coprod _{g\geq 0} B \cD_g
\] 
is an algebra over  $\cW^{2k}.$
By \autoref{thm:even_mfld_stability}, $\cW^{2k}$  is an operad with homological stability (OHS) and hence by our main result, \autoref{loopsandaction}, the group completion
of any $\cW^{2k}$-algebra is an infinite loop space. By \autoref{htpycom_detail}, the group completion is given by the plus construction on the homotopy colimit,
as the underlying monoid is homotopy commutative.
\end{proof}

More generally, 
the forgetful maps to diffeomorphism groups factor through the orientation preserving diffeomorphism groups. Thus we can also consider the oriented analog $\mathcal D^+_g$ of $\mathcal D_g$.
Galatius and Randal-Williams \cite {GRWII} identify the associated limit spaces in these cases with certain homotopy quotients
\[
(\Omega ^\infty {\bf MT} \theta)//\hAut (u)
\quad \quad \text { and } \quad \quad
(\Omega ^\infty {\bf MT} \theta)// \hAut (u^+)
\]
of $\Omega ^\infty {\bf MT} \theta$ by the monoid of weak homotopy equivalences of the $k$-connected cover $B$ of $BO(2k)$ which commute with the maps
 $u\colon B\to BO(2k)$ and $u^+\colon B\to BSO(2k)$ respectively.
We note here that there is no a priori reason why these quotient spaces are infinite loop spaces. 

Other tangential structures are also considered in \cite{GRWII} and could be considered similarly.

\begin{rmk}\label{enlarge_Dg}
The category $\cD_g$ (and hence $\pi_0 \cD_g$ considered below) could be enlarged by introducing further atomic manifolds $A_\alpha \colon Q \to Q$ with $\alpha$ from some countable set $S$. 
The union of the  associated classifying spaces form $\cW^{2k} $-algebras and hence group complete to infinite loop spaces.
\end{rmk}

\subsection{Mapping class groups}

Let $\Gamma W _{g} \coloneqq \pi_0 (\Diff (W_{g,1};\partial ))$ denote the mapping class group of $W_{g,1} $.
By replacing the diffeomorphism groups by their connected components we can construct  a new  category $\pi_0 \cD_g$ with a canonical forgetful functor $\cD_g \to \pi_0 \cD_g$ so that 
\[
\coprod_{g\geq 0} B\pi_0 \cD_g
\]
is an algebra over
$\cW^{2g}$. We note the homotopy equivalence $B\pi_0 \cD_g \simeq B\Gamma W_{g, 1}$.

\begin{cor}
The space $\mathbb Z \times  (\hocolim _{g\to \infty}B\Gamma W_{g,1}) ^+$
is an infinite loop space.
\end{cor}

\begin{proof}
The argument is precisely as above.
\end{proof}

Mapping class groups of surfaces and three manifolds are well-studied and notoriously complicated. However, in higher dimensions these mapping class groups  tend to be much better behaved. In the next section we will identify the above infinite loop space with a K-theory space when $2k=4$. Nevertheless, as we will argue,  the above corollary and \autoref{enlarge_Dg} provide us with many examples of infinite loop spaces where at the moment we do not have any other means of identifying them as such.

\subsection{The map to \texorpdfstring{$K$}{K}-theory}
Our final goal is to construct the map from $\Omega ^\infty \text{\bf MT} \theta$ to $K$-theory as an infinite loop space map.

Diffeomorphisms act on homology. In particular, for $W_{g,1}$ of dimension $2k$ the diffeomorphisms act on the middle dimensional homology group
\[H_k(W_{g,1}) \simeq \mathbb Z ^{2g} \simeq H_k (W_g)\] 
preserving the intersection form.
Define categories $\mathcal K _{g}$ with objects the same as in $\cD_g$ and morphisms from $W$ to $W'$ given by
linear isomorphisms  $H_k (W) \to H_k(W')$ that preserve the natural intersection forms. Thus when $k$ is odd, $B\mathcal K_g \simeq B\text{Sp}_{2g} \mathbb Z$, and when $k$ is even,  $B\mathcal K_g \simeq B O_{g,g} \mathbb Z$. There is a chain of  functors
\[
\cC_{g,0} \longrightarrow \cD_{g} \longrightarrow \pi_0 \cD_g \longrightarrow \mathcal K_g
\]
giving rise to a chain of $\cW^{2k}$-algebra maps on the union over all $g$ of the classifying spaces, and hence maps of infinite loop spaces on their group completions.
We summarize with the following theorem.

\begin{thm}\label{infmap}
For $2k\geq 2$, a fibration $\theta \colon B \to BO(2k)$ with $B$ $k$-connected, and an operadic choice of 
$\theta$-structure on $Q= S^{2k-1}$,
the  action of the diffeomorphisms on the middle dimensional homology of the underlying manifolds induces a map of infinite loop spaces from
\[
\mathbb Z \times (\hocolim\limits_{g \to \infty} \cM^\theta_k (W_{g,1}, \ell_{W_{g,1}}))^+ \simeq \Omega ^\infty {\bf MT} \theta
\]
factoring through the infinite loop spaces
\[
\mathbb Z \times (\hocolim _{g\to \infty} B\Diff (W_{g,1}, \partial ))^+ \text {  and  } \,\,
\mathbb Z \times (\hocolim _{g\to \infty} B \Gamma W_{g,1})^+
\]
to algebraic $K$-theory, $K\text{Sp} \simeq \mathbb Z \times B\text{Sp}  (\mathbb Z)^+$ when $k$ is odd, and 
$K \text O \simeq \mathbb Z \times B \text {O}_{\infty, \infty} (\mathbb Z)^+$ when $k$ is even. 
\qed 
\end{thm}

We finish this section with a discussion of the map from the infinite loop space associated to the mapping class groups of the $W^{2k}_{g,1}$ to $K$-theory.

For $2k=4$: the stable mapping class group 
\[\Gamma W^{2k}_\infty \coloneqq \lim _g \Gamma W^{2k}_{g,1}
\] 
is isomorphic to the group $O_{\infty, \infty} (\mathbb Z)$. Indeed, following \cite {Giansiracusa},
we can give the following argument. By a theorem of Wall  \cite{Wall}, the map 
\[\Gamma W^4_{g,1} \to O_{g,g}(\mathbb Z)\] is surjective for all $g\geq1$. This map factors through the group of pseudo-isotopy classes $\mathcal P (W^4_{g,1})$, and by a theorem of Kreck \cite {Kreck} the map 
\[\mathcal P(W^4_{g,1}) \to O_{g,g} (\mathbb Z)\] is injective. Finally, by a theorem of Quinn \cite {Quinn},
pseudo-isotopic diffeomorphisms will be isotopic after stabilization with a number of copies of $S^2 \times S^2$. 

For $2k>4$:
While both Wall's and Quinn's theorems are restricted to dimension four, a version of Kreck's theorem holds in all even dimensions, and the mapping class group of $W^{2k}_{g,1}$ is an extension of the 
automorphism group of the quadratic forms $Q_{W^{2k}_{g}}$ associated to the bilinear intersection form on $H_k (W_{g}^{2k}) \simeq \mathbb Z ^{2g}$:
\[
1 \longrightarrow \mathcal I^{2k}_{g,1} \longrightarrow \Gamma W^{2k}_{2g,1} \longrightarrow \text {Aut} (Q_{W^{2k}_{g}} )\longrightarrow 1.
\]
The difference from the case where $2k=2$ is that the analog of the Torelli group $\mathcal I^{2k}_{2g,1}$  in higher dimensions is much simpler. According to Kreck \cite{Kreck}, it is an extension of 
two abelian groups, one of which is the finite group $\Theta _{2k+1}$ of exotic smooth structures on the $2k+1$ sphere  and,  depending on 
 $k$ (mod 8), the other is free abelian of size $\mathbb Z^{2g}$ or all 2-torsion:
\[
1 \longrightarrow \Theta _{2k+1} \longrightarrow 
\mathcal {I}^{2k}_{g,1} 
\longrightarrow \text {Hom} ( H_k (W_{g,1}^{2k}), S\pi_k (SO(k))) \longrightarrow 1.
\]
Here $S\pi_k(SO(k))$ is the image of $\pi_k (SO(k))$
under the map $SO(k) \to SO (k+1)$.
We refer to \cite {GRWmcg} for a recent study of these mapping class groups and more details including a table of the $S\pi_k(SO(k))$. In particular, it is proved in \cite {GRWmcg} that 
\[
\Gamma W^{12}_{g,1} \simeq \mathbb Z/3\mathbb Z \times
O_{g,g} (\mathbb Z).
\]
Thus in this case, the mapping class groups give rise to the infinite loop space 
\[
\mathbb Z \times (B\Gamma ^{12}_\infty )^+ \simeq B\mathbb Z/3\mathbb Z \times KO.
\]
In general, the two extensions from which $\Gamma W^{2k}_{g, 1}$ is built are not well understood. Thus, the infinite loop space structure on 
\[
\mathbb Z \times (B\Gamma W^{2k} _\infty )^+
\]
may not be surprising but seems also non-obvious and new in most cases.

\appendix
\section{Surface Operads}\label{sec:examples}

In this appendix we collect examples of operads with homological stability arising from the geometry of surfaces. 
The relevant homological stability results are  due to Harer
\cite {Harer}, Wahl \cite {wahl},  and Randal-Williams \cite{RWstability}. 

There are also homological stability results related to graphs and automorphisms of free groups. The operads related to these groups tend to  receive a map from an $E_\infty$-operad, and hence it is not surprising that their algebras group complete to infinite loop spaces. Therefore we have  not treated these examples in detail here.

\subsection{Surface operads}
In \cite{T}, a version of our main theorem was proved for a particular surface operad.  
In that set-up it was necessary to have a  strict multiplication, that is a map from $Ass$ into the surface operad. To construct this map some surfaces  had to be identified and diffeomorphisms were replaced by 
mapping class groups. The present set-up gives us more flexibility and the constructions can be significantly simplified.

\begin{eg} [Oriented surfaces and diffeomorphisms]\label{ex:oriented}
  We will now construct a variant of the surface operad studied in \cite {T}. Like in the constructions of \S\ref{highdim}, we have three atomic surfaces: a disk $D$, a torus $T=F_{1,2}$ and a pair of pants $P=F_{0,3}.$  (In \S\ref{highdim}, $T$ is denoted by $H$, and $P$ is denoted by $M$.)
  Let $F_{g,n+1}$ denote an oriented surface built out of atomic surfaces that has genus $g$ with $n+1$ (collared) boundary components, $n$
incoming and 1 outgoing. 

  We define a connected groupoid $\cS_{g,n+1}$ where the objects are surfaces of type $F_{g, n+1}$ 
  with a labeling of the incoming boundary components. 
  Define
  \[
   \cS_{g,n+1}((F,\sigma),(F',\sigma'))\coloneqq  \textrm{Diff}(F,F';\partial), 
  \]
  the diffeomorphisms that preserve the collars and the labels on the boundary components. 
  We have 
  a homotopy equivalence 
  \[
   B\cS_{g,n+1}\simeq B\Gamma_{g,n+1},
  \]
where $\Gamma_{g,n+1}$ is the mapping class group.

  We can now define an operad $\cS$ with 
  \[
   \cS(n)=\coprod_{g \geq 0} B\cS_{g,n+1}.
  \]
  The gluing of surfaces defines associative and equivariant operad structure maps.
  The operad is graded by the genus $g \in \mathbb N$. Consider the suboperad $\cS_0$ 
  corresponding to the genus zero surfaces.
  To identify an $A_\infty$-suboperad $\cA$ in $\cS_0$, connect the point corresponding to $1 \in S^1$ in the outgoing
  boundary circle by a simple path to those in each of the incoming boundary circles of the pair of pants $P$.  See \autoref{fig:trees}. 
  This way, after gluing,  every surface of genus zero will be decorated by a tree connecting the outgoing
  boundary circle to each of its incoming circles. Form $\cA \subset \cS_0$ by restricting the diffeomorphisms to those that 
  also map trees to trees. The complement of a tree is a disk, so these diffeomorphism groups will
  be contractible, and hence $\cA$ is an $A_\infty$-operad. 
  
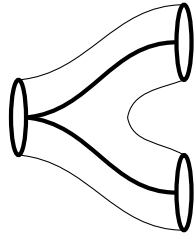
\begin{figure}
\begin{tikzpicture}
  \draw [ultra thick] (5.3,1) ellipse (.10 and .5);
  \draw [ultra thick] (7.5,0) ellipse (.10 and .5);
  \draw [ultra thick] (7.5,2) ellipse (.10 and .5);
  \draw (5.3,1.5) to [out=5,in=180]  (7.5,2.5);
  \draw (5.3,.5) to [out=-5,in=180]  (7.5,-.5);
  \draw (7.5,1.5) to [out=200,in=80] (6.75,1)  to [out=-80,in=150]  (7.5,.5);
  \draw [ultra thick](5.4, 1) to [out=-5,in=180]  (7.4,0);
  \draw [ultra thick] (5.4, 1) to [out=5,in=180] (7.4,2);
\end{tikzpicture}
\caption{Trees within surfaces}
\label{fig:trees}
\end{figure}

  By the main theorem of \cite {Madsen-Weiss} the associated infinite loop space is
  \[
   \g \cS(0) \simeq \mathbb Z\times B\Gamma_\infty ^+ \simeq \Omega ^\infty {\bf MT} SO(2),
  \]
and by \cite{Wahl04} the explicit infinite loop space structure on the right is the same as the one induced by the operad structure on $\g \cS(0)$.
\end{eg}

\begin{eg}[Nonorientable surfaces and diffeomorphisms]
  We now add another surface building block: a nonorientable surface of genus one $N=\mathbb{R}P^2 \setminus (D^2 \coprod D^2)$.
  Let $N_{k,n+1}$ be a surface of nonorientable 
  genus $k$ with one outgoing  and $n$ incoming boundary components.  As in \autoref{ex:oriented}
we construct $N_{k,n+1}$ from $D$, $P$, $S^1$ and  $N$,  and associate a groupoid $\cN_{k, n+1}$ where the morphisms are given by diffeomorphisms. 

  Thus we define an operad $\cN$ with
  \[
    \cN( n) \simeq \coprod_{k\geq 0} B \cN _{k, n+1}.
  \]
  The genus zero suboperad $\cN_0$ is the same as $\cS_0$, so we have a map
  $\mu \colon\cA \to \cN$ satisfying condition \ref{htpycom}.
  By  \cite[A]{wahl} we have 
  $H_*(\cN_{k, n+1}) \cong H_*(\cN_{k,1})$ for $k\geq 4\ast+3.$ 
 Hence, $\cN$ is an OHS.
 By \cite{wahl}, the associated infinite loop space is
  \[
    \g \cN (0) \simeq \mathbb Z \times B\cN_\infty^+ \simeq \Omega ^\infty {\bf MT} O(2),
  \]
  where $\cN_\infty = \lim _{k\to \infty } \pi_0 \Diff (N_{k, 1}, \partial)$ denotes the infinite mapping class group.
\end{eg}

\begin{eg}[Mixed surfaces and diffeomorphisms]
  We include this example as it has a grading different from $\mathbb N$.
  The procedure in \autoref{ex:oriented}, applied to the atomic surfaces $D$, $P$, $T$, $S^1$ and $N$, 
  constructs an operad we will denote by  $\cS \cN$.  Considering both orientable and nonorientable genus 
  defines a $\{+, - \} \times \mathbb N$ grading on $\cS\cN$ 
  with addition defined by $ (+, g) + (+, g') = (+, g+g')$, $ (-, k) + (-, k') = (-, k+k')$ and
  $(+, g) + (-, k) = (- , 2g +k)$. The group completion of this monoid is $\mathbb Z$ with generator $(-, 1)$.

  Take $N$ to be the propagator for $\cS\cN$ and observe that  
  $\gamma(T;N)=\gamma(N;\gamma(N;N))$.  In particular, under the action of this propagator,
orientable and nonorientable surfaces conflate and  
  \[
   \g \cS \cN(0) \simeq \mathbb{Z} \times B\cN^+_\infty.
  \] 

As a consequence, this operad gives a bridge between the operads of orientable surfaces and nonorientable surfaces. Comparing building blocks, there are operad maps 
  \[
   \cN \to \cS \cN \leftarrow \cS
  \]
  and the maps on connected components of group completions correspond to 
  \[
    \mathbb{Z} \xrightarrow{\id} \mathbb{Z} \xleftarrow{\cdot 2} \mathbb{Z}.
  \]
\end{eg}

\begin{eg}[Framed, $r$-spin, and pin surfaces]
In \cite{RWstability} Randal-Willimas studies the homological stability for diffeomorphism groups of surfaces with general $\theta$-structures. In particular, for $r$-spin structures  $\theta ^r$ on oriented surfaces, he shows that reducing boundary components by gluing in disks induces isomorphisms in degrees increasing with the genus of the underlying surfaces, see \cite[2.14]{RWstability}. Hence, these diffeomorphism groups (or their associated  mapping class groups) give rise to OHSs.
Similarly, he shows that homological stability holds for the $pin ^+$ and $pin^-$
mapping class groups of surfaces, see \cite[4.18, 4.19]
{RWstability}, and thus these too give rise to OHSs.  In all these cases care has to be taken when defining the $\theta$ structures on the atomic surfaces so that gluing is well defined and gives rise to an operad.
\end{eg}

\subsection{Applications}

\begin{eg}
[Map to \texorpdfstring{$K$}{K}-theory]
In \cite{T} Tillmann constructs an infinite loop space map into $K$-theory using the operad associated to orientable surfaces. There is also an interesting example coming from the unorientable case.
In particular, let ${N}_{k,1}$ be a nonorientable surface of genus $k$ with 1 boundary component. 
  The action of the mapping class group $\cN_{k,1}$ on $H_1({N}_{k,1})=\mathbb{F}_2^k$ induces a representation 
  \[
   \rho\colon \cN_{k,1} \to GL_k \mathbb{F}_2.
  \]
  Note that $\cN(0) = \coprod_{k \geq 0} B \cN_{k,1}.$
  Thus the map $\rho\colon  \cN(0) \to \coprod_{k \geq 0} BGL_k \mathbb{F}_2$ is a map of $\cN$-algebras,
  and hence the map on group completions is an $\Omega^\infty$-space map
  \[
    \rho\colon \mathbb{Z} \times B\cN_\infty^+ \simeq \Omega^\infty \mathbf{MT}O(2) \to \mathbb{Z}\times BGL(\mathbb{F}_2)^+ \simeq K(\mathbb{F}_2).
  \]
  Note that the Madsen-Tillmann spectrum $\Omega ^\infty \mathbf{MT}O(2)$ is geometrically defined as a Thom spectrum, 
  while $K(\mathbb{F}_2)$ is algebraically defined. It would be difficult to define a map of infinite loop spaces directly without 
  making use of the Madsen-Weiss Theorem which describes $\Omega ^\infty \mathbf{MT} O(2)$ in terms of classifying spaces of diffeomorphisms groups of surfaces.

\end{eg}

\begin{eg}[Detecting $\Omega^\infty$-spaces]
  Let $\tilde{\cN}_{k,1}$ be defined as the split extension of $\cN_{k,1}$ by its $\mathbb {F}_2$-homology $H_1(N_{k,1}, \mathbb {F}_2)=\mathbb{F}_2 ^k$:
  \[
   0 \longrightarrow \mathbb F_2 ^{k} \longrightarrow \tilde {\cN}_{k,1} \longrightarrow \cN_{k,1} \longrightarrow 0.
  \]
  Then there exists an $\cN$-algebra
  \[
    X \simeq \coprod _{g\geq 0} B \tilde{\cN}_{k,1}
  \] 
  and $\mathbb{Z} \times B \tilde{\cN}_\infty^+$ is an $\Omega^\infty$-space. 
  At this point we do not know another description of this infinite loop space.

  Similarly, as already noted in \cite {T}, in the oriented case we get an infinite loop space structure on the group completion of 
  \[
  Y= \coprod _{g\geq 0} B\tilde \Gamma _{g, 1}.
  \]
and Ebert and Randal-Williams show in \cite {Ebert-RW} 
  \[
\Omega B (Y) \simeq    \Omega \mathbf{MT} SO(2) \wedge B\mathbb C^\times _+.
  \]
\end{eg}

\bibliography{Operads}{}
\bibliographystyle{amsalpha}
\end{document}